\pdfoutput=1
\documentclass[
	bigMargin = false,
	font = palatino,
	final
	]{MyClass}

\newtheorem{theorem}{Theorem}[section]
\newtheorem{lemma}[theorem]{Lemma}

\theoremstyle{definition}
\newtheorem{definition}[theorem]{Definition}
\newtheorem{example}[theorem]{Example}
\newtheorem{corollary}[theorem]{Corollary}

\newtheorem{proposition}[theorem]{Proposition}

\theoremstyle{remark}
\newtheorem{remark}[theorem]{Remark}

\numberwithin{equation}{section}

\newcommand{\abs}[1]{\lvert#1\rvert}

\newcommand{\tensorProd}  {\otimes}

\DeclarePairedDelimiterX\dualPair[2]{\langle}{\rangle}{#1,#2} 
\newcommand{\dif}{\mathop{}\!\mathrm{d}}    
\newcommand{\difStrat}{\mathop{}\!\circ \! \: \mathrm{d}}    
\newcommand{\isomorph}{\simeq}
{}
\NewDocumentCommand{\setSymbol}{o}{
  \nonscript \, : 
  \allowbreak
  \nonscript \,
  \mathopen{ }
}
\DeclarePairedDelimiterX\set[1]{\{}{\}}{        
  
  #1
}
\DeclarePairedDelimiter\equivClass{[}{]}				

\DeclareMathOperator{\id}{id}
\DeclareMathOperator{\Hom}{Hom}

\DeclareMathOperator{\tr}{tr}

\DeclareMathOperator{\End}{End}
\DeclareMathOperator{\GL}{GL}

\newcommand{\TBundle}{\mathrm{T}}
\newcommand{\CotBundle}{\mathrm{T}^*}

\newcommand{\SOGroup}{\mathrm{SO}}
\newcommand{\OGroup}{\mathrm{O}}
\newcommand{\UGroup}{\mathrm{U}}
\newcommand{\UAlgebra}{\operatorname{\LieA{u}}}
\newcommand{\SUGroup}{\mathrm{SU}}

\newcommand{\SpGroup}{\mathrm{Sp}}
\newcommand{\SpAlgebra}{\operatorname{\LieA{sp}}}
\newcommand{\R}{\mathbb{R}}
\newcommand{\LieA}[1]{\mathfrak{#1}}
\newcommand{\T}{\mathsf{T}}								

\NewDocumentCommand{\smallMatrix}{m}{\ensuremath{\left(\begin{smallmatrix}#1\end{smallmatrix}\right)}}
\newcommand{\intersect}{\cap}
\NewDocumentCommand{\field}{m}{\mathbb{#1}}			
\DeclareDocumentCommand{\C}{}{\field{C}}			
\NewDocumentCommand{\Z}{}{\field{Z}}				
\DeclareDocumentCommand \I { } { 			
	\mathrm{i}
}

\DeclareMathOperator{\AdAction}{Ad}

\DeclareMathOperator{\Expect}{\mathbb{E}}
\DeclarePairedDelimiterX\scalarProd[2]{\langle}{\rangle}{#1,#2}	

\Title{Expectation values of polynomials and moments on general compact Lie groups}

\Abstract{We develop a powerful framework to calculate expectation values of polynomials and moments on compact Lie groups based on elementary representation-theoretic arguments and an integration by parts formula. In the setting of lattice gauge theory, we generalize expectation value formulas for products of Wilson loops by Chatterjee and Jafarov to arbitrary compact Lie groups, and study explicit examples for many classical compact Lie groups and the exceptional Lie group $G_2$. Extending classical results by Collins and Lévy, we use our framework to derive expectation value formulas of polynomials of matrix coefficients under the Haar measure, Brownian motion, and the Wilson action. In particular, we construct Weingarten functions for general compact Lie groups by studying the underlying tensor invariants, and apply this to $\SUGroup(N)$ and $G_2$.
}



\Institution{delftshanghai}{
	Institute of Applied Mathematics,
	Delft University of Technology,
	The Netherlands
	and
	School of Mathematical Sciences,
	Shanghai Jiao Tong University,
	China}
\Institution{delft}{
	Institute of Applied Mathematics,
	Delft University of Technology,
	The Netherlands}

\Author[
	affiliation = delftshanghai,
	email = research@tobiasdiez.de,
	]{diez}{Tobias Diez}
\Author[
	affiliation = delft,
	email = {l.t.miaskiwskyi@tudelft.nl}
	]{miaskiwskyi}{Lukas Miaskiwskyi}

\begin{document}

\vspace*{-4em}
\MakeTitle
\tableofcontents

\section{Introduction}

Integration over Lie groups plays a central role in many areas of mathematics and theoretical physics.
It lies at the core of random matrix theory and has become an important tool to describe a wide range of physical systems including lattice gauge theory \parencite{Weingarten1978}, quantum chaotic systems \parencite{CotlerHunterJonesLiuEtAl2017}, many-body quantum systems \parencite{GuhrMuellerGroelingWeidenmueller1998}, quantum information theory \parencite{CollinsNechita2016} and matrix models for quantum gravity and Yang--Mills theory in two dimensions \parencite{DiFrancescoGinspargZinnJustin1995,Xu1997}.
In this paper, we develop a general framework to calculate expectation values of polynomials of group elements and their inverses on a compact Lie group \( G \) of the form
\begin{equation}
	\label{eq:intro:polynom}
	\int_G \tr_\rho (c_1 g^{\pm 1} \dotsm c_n g^{\pm 1}) \, \nu(g) \dif g, \qquad c_i \in G,
\end{equation}
where \( \dif g \) denotes the normalized Haar measure, \( \nu \) is a probability density and the trace is taken in a given representation \( \rho \) of \( G \).
Expanding the integrand, the problem reduces to a computation of the expectation value of the so-called moments
\begin{equation}
	\label{eq:intro:moments}
	\int_G g^{}_{i^{}_1 i'_1} \dotsm g^{}_{i^{}_p i'_p} \, g^{-1}_{j^{}_1 j'_1} \dotsm g^{-1}_{j^{}_q j'_q} \, \nu(g) \dif g ,
\end{equation}
where \( g_{ij} = \rho(g)_{ij} \) are the matrix entries of \( g \in G \) in the representation \( \rho \).

Given the numerous applications, these integrals are well-studied in the literature.
For matrices drawn randomly from the Haar distribution (\( \nu = 1 \)), the calculation of the moments has been initiated by theoretical physicist \textcite{Weingarten1978} motivated by problems in lattice gauge theory.
\Textcite{Collins2002} developed a rigorous mathematical framework for computing moments for the unitary group, which has been extended to the orthogonal and symplectic group by \textcite{CollinsSniady2006}.
In the unitary case, the developed Weingarten calculus expresses the integral~\eqref{eq:intro:moments} as a sum over so-called Weingarten functions, which are functions defined on the symmetric group.
This approach makes heavy use of representation theory in the form of Schur--Weyl duality.
Recently, the Weingarten calculus has been rephrased in terms of Jucys--Murphy elements~\parencite{Novak2010,ZinnJustin2010,MatsumotoNovak2013}.
A special but important case is the computation of joint moments of traces of powers of group elements, that is, essentially integrals of the form~\eqref{eq:intro:polynom} with all coefficients \( c_i \) set to the identity element.
For the unitary group, this has been extensively studied in \parencite{DiaconisShahshahani1994,DiaconisEvans2001} where it was used to obtain central limit theorems of eigenvalue distributions; see also \parencite{PasturVasilchuk2004,HughesRudnick2003} for analogous results for other groups.

Lattice gauge theories are a natural area where integrals over Lie groups play a central role.
They were originally introduced by Wilson as discrete approximations to quantum Yang--Mills theory and evolved to become one of the most promising approaches to study non-perturbative effects in QCD such as quark confinement.
The most important gauge-invariant observables in lattice theory are Wilson loops (traced holonomies) whose long range decay serves as an indicator for the confining behavior.
In the Euclidean formulation of the theory, the expectation value of a Wilson loop is an integral of the form~\eqref{eq:intro:polynom} with the probability density \( \nu \) being a Boltzmann weight relative to the Yang--Mills action.
However, explicitly computing expectation values of and correlations between Wilson loops is a difficult, if not impossible, challenge.
One thus usually resorts to numerical methods such as classical Monte Carlo simulation to approximate such integrals.

Indeed, an analytical understanding of the Wilson loop expectation values in the continuum and infinite volume limit is an essential ingredient to solve the Yang--Mills mass gap problem, which is one of the seven Millennium Problems posed by the Clay Mathematics Institute.
\Textcite{Hooft1974} realized that, when the rank \( N \) of the gauge group tends to infinity, the theory simplifies in many ways and can be solved analytically in certain cases.
In particular, the Wilson loops then satisfy the Makeenko--Migdal equations and factorize, \ie, the expectation value of a product of Wilson loops equals the product of the expectation values of the individual Wilson loops.
Recently, \textcite{Chatterjee2019} established in the large \( N \)-limit of \( \SOGroup(N) \) lattice gauge theory an asymptotic formula for expectation values of products of Wilson loops in terms of a weighted sum of certain surfaces.
These surfaces are defined starting from the collection of loops using the four operations of merging, splitting, deformation and twisting.
The proof proceeds by a complicated and lengthy calculation which hinges on Stein's method for random matrices.
Analogous results have been obtained for \( \SUGroup(N) \) in \parencite{Jafarov2016} using similar methods.
This development sparked renewed interest, leading to further progress for large-\( N \) gauge theories \parencite{ChatterjeeJafarov2016,BasuGanguly2018,Chatterjee2021}.

Another approach to a rigorous definition of a quantum Yang--Mills theory is the construction of the Yang--Mills measure and thus of the path integral using Brownian motion on the structure group.
This direction has been pioneered by \textcite{Driver1989,GrossKingSengupta1989} in two dimensions.
We refer the reader to \parencite{Sengupta2008} for a relative recent review of two-dimensional Yang--Mills theory.
In the physics literature, the Yang--Mills measure is taken to be the Lebesgue measure on the space of connections, weighted by a Boltzmann density involving the Yang--Mills action.
To make sense of this formal description, one usually uses the holonomy mapping to define the Yang--Mills measure in terms of group-valued random variables indexed by embedded loops whose distribution is given by the heat density (at a ``time'' proportional to the area enclosed by the loop).
For this reason, the calculation of expectation values of polynomials on Lie groups with respect to a Brownian motion attracted a lot of attention, especially in the large-\( N \) limit.
In particular, combinatorial integration formulas for the expectation values of polynomials under the heat kernel measure have been obtained by \textcite{Xu1997,Levy2008} and recently generalized by \textcite{Dahlqvist2017} to also allow polynomials in inverses of group elements.

The study in the aforementioned papers rely on heavy machinery from representation theory in the form of Schur--Weyl duality or Jucys--Murphy elements, or on a detailed probabilistic analysis using for example Stein's method.
Due to this complexity, the results have usually been obtained first for the unitary group, and then generalized in subsequent papers to other groups such as the orthogonal or symplectic group.
Moreover, the methods have been tailored to the specific probability measure under study which made it hard to transfer progress from one scheme to another.
In contrast, we here deduce and extend the main results of these papers from an elementary integration by parts formula.
This allows us to generalize these results to arbitrary compact Lie groups and to analyze the Haar, Wilson and heat kernel cases simultaneously and on equal footing.

Our first main result is \cref{prop:theoremA} which describes the expectation value of a product of Wilson loop observables in terms of other Wilson loops that are obtained from the initial family through two operations that we call twisting and merging\footnotemark{}.
\footnotetext{For the unitary group, these operations correspond to the Fission and Fusion processes of \parencite[Section~III]{Samuel1980}, respectively.}
This is a generalization of the results of \textcite{Chatterjee2019,Jafarov2016} to arbitrary compact Lie groups, arbitrary probability measures and arbitrary group representations.
Our construction shows that the operations of twisting and merging are determined by an operator in the universal enveloping algebra of the Lie algebra that can be seen as an operator-theoretic counterpart to the so-called completeness relations.
Moreover, these two operations can be represented in a diagrammatic way that resembles the Feynman path integrals rules.
This diagrammatic calculus is similar but different to the one developed by \textcite{BrouwerBeenakker1996} for the unitary group, \cf also \parencite{Cvitanovic1976}.
For the Haar measure and for the Brownian motion, the resulting equations for the expectation of a product of Wilson loops lead to a recursive formula that can be solved using a straightforward algorithm.
In the case of the Haar integral over the unitary group, we recover the recursion relations given in \parencite[Section~III]{Samuel1980}. 
For the Yang--Mills Wilson action, the equation takes a relative simple form which does not involve merging of Wilson loops with plaquette operators as in \parencite{Chatterjee2019} and which has the additional benefit to reduce the operations on the family of loops needed from four to two.
Moreover, in the case of the unitary group, the structure of the equation is particularly well-suited to a large-\( N \) limit.
As applications of our general framework, the result for the groups \( \SOGroup(N) \), \( \SpGroup(N) \), \( \UGroup(N) \) , and \( G_2 \) are discussed in more detail in \cref{sec:examples}.

In the second part of the paper, we investigate the moment integrals~\eqref{eq:intro:moments} for an arbitrary compact Lie group.
\Cref{prop:matrixCoefficients:general} shows that the moments satisfy an eigenvalue equation whose particular form depends on the probability measure.
For the Haar measure, the moments yield a projection onto the subspace of invariants.
Specialized to the unitary group, this result is a restatement of the well-known fact that the moments yield a conditional expectation onto the group algebra of the symmetric group \parencite[Proposition~2.2]{CollinsSniady2006}.
Moreover, \cref{prop:matrixCoefficients:tensorRepSchur} yields an explicit expansion of the moments as a sum over a spanning set of invariants.
In particular, we define a Weingarten map for every compact Lie group (depending on the group representation and on the spanning set of invariants) and show that it gives the coefficients in this expansion of the moments.
This is similar in sprit to the definition of the Weingarten map as a pseudoinverse in \parencite{ZinnJustin2010} and equivalent to the results of \parencite{Collins2002,CollinsSniady2006} for the unitary, orthogonal and symplectic group.
As a novel application, we determine in \cref{sec:examples:g2group} the Weingarten map, and thus the moments, for the exceptional group \( G_2 \) in its natural \( 7 \)-dimensional irreducible representation.
In the case of Brownian motion, the moments can be calculated using the eigenvalues of the Casimir operator and converge for large times to the moments with respect to the Haar measure, see \cref{prop:matrixCoefficients:brownian}.
This is a refinement and extension of the results of \textcite{Levy2008,Dahlqvist2017}, where only the groups \( \UGroup(N) \), \( \OGroup(N) \), and \( \SpGroup(N) \) were considered.

As we have mentioned above, at the heart of our approach lies a simple integration by parts formula.
This is perhaps most similar to the derivation of the Schwinger--Dyson equation for the Gaussian unitary ensemble, see, \eg, \parencite[Equation~(5.4.15)]{AndersonGuionnetZeitouni2010}.
To illustrate how integration by parts can be used to calculate moments, consider the simple example of \( T_{ijkl} = \int_G g_{ij} g^{-1}_{lk} \, \dif g \).
After inserting the Laplacian in the first factor, integration by parts yields
\begin{equation}
	\int_G (\Delta g_{ij}) g^{-1}_{lk} \, \dif g = - \int_G \dualPair{\dif g_{ij}}{\dif g^{-1}_{lk}} \, \dif g .
\end{equation}
The Schur-Weyl lemma implies that \( g_{ij} \) is an eigenvector of the Laplacian with, say, eigenvalue \( \lambda \).
Thus, the left-hand side equals \( \lambda T_{ijkl} \).
On the other hand, the right-hand side can be calculated by using an orthonormal basis \( \xi^a \) for the Lie algebra \( \LieA{g} \) and introducing the operator \( K_{ijkl} = \xi^a_{ij} \xi^a_{kl} \) (implicitly summing over \( a \)).
In summary, we obtain
\begin{equation}
	\lambda T_{ijkl} = K_{rjls} T_{irks}.
\end{equation}
For the fundamental representation of \( G = \UGroup(N) \), the completeness relation implies \( K_{rjls} = - \delta_{rs} \delta_{jl} \) and thus 
\begin{equation}
	\lambda T_{ijkl} = - \delta_{jl} T_{irkr} = - \delta_{jl}\delta_{ik},
\end{equation}
where the second equality follows from the definition of \( T_{ijkl} \).
The calculation of matrix coefficients of higher degree involves more delicate combinatorics, but the strategy remains the same: integration by parts yields an eigenvalue equation for the matrix coefficients, which is then solved by using a group-dependent completeness relation.


\paragraph*{\textbf{Acknowledgments}}\mbox{}\\
We are very much indebted to Alessandra Cipriani, Bas Janssens and Richard Kraaij for many helpful discussions and for organizing a reading seminar on probabilistic aspects of lattice gauge theory, thereby inciting our interest in this topic.
We gratefully acknowledge support by the NWO grant 639.032.734 \textquote{Cohomology and representation theory of infinite dimensional Lie groups}.

\section{Setting}
\subsection{Differential Geometry on Lie Groups}
In all that follows, we are using the Einstein summation convention where repeated indices are automatically summed over. Occasionally, this convention will be overridden by explicit summation symbols, when we want to more closely specify the range of summation.

For the rest of the paper, consider the following setting: We are given a compact Lie group $G$ and a finite-dimensional, complex, irreducible representation $\rho : G \to  \GL(V)$. For all $g \in G$, write
\begin{equation}
\tr_\rho(g) := \tr(\rho(g)).
\end{equation} 
Equip the Lie algebra \( \LieA{g} = T_e G \) of \( G \) with an \( \AdAction_G \)-invariant positive-definite symmetric bilinear form \( \kappa \).
In the examples we consider, $G$ is usually semisimple and \( \kappa \) a negative multiple of the Killing form.
By translation, the inner product \( \kappa \) induces a canonical bi-invariant Riemannian metric $\langle \cdot, \cdot \rangle$ on \( G \).
We normalize the volume form so that \( G \) as unit volume, and denote the corresponding probability measure by \( \dif g \).  

We fix an orthonormal basis $\{\xi^a \in \mathfrak{g}\}$ of $\LieA{g}$ with respect to \( \kappa \), and use it to define a global frame of $TG$ by left translation. Analogously, the dual basis $\{ \epsilon^a \in \mathfrak{g}^*\}$ associated to $\{\xi^a \}$ defines a global frame of the cotangent bundle $T^* G$.

We collect some abuses of notation: We will denote the global frames of $TG$ and $T^*G$ with the same letters $\xi^a, \epsilon^a$ as the pointwise objects. 
Similarly, we identify elements of \( \LieA{g} \) with left-invariant vector fields on \( G \) and with derivations on \( C^\infty(G) \).  
Further, the Lie group representation $\rho$ induces a Lie algebra representation $\mathfrak{g} \to \End(V)$ which we will denote by the same letter $\rho$.

The Riemannian metric naturally induces the \emph{musical isomorphisms}
\begin{align}
\flat &: TG \to T^*G, \quad v_p \mapsto \langle v_p, \cdot \rangle, \\
\sharp := \flat^{-1} &: T^*G \to TG .
\end{align}
We extend the inner product on the fibers of \( \TBundle G \) to \( \CotBundle G \) by declaring
\begin{equation}
	\scalarProd{\alpha}{\beta} = \scalarProd{\alpha^\sharp}{\beta^\sharp}
\end{equation}
for \( \alpha, \beta \in \CotBundle G \) in the same fiber.

The \emph{Laplace--Beltrami operator} is defined by
\begin{equation}
	\label{eq:DefinitionLaplace}
\Delta: C^\infty(G) \to C^\infty(G), \, f \mapsto \Delta f := \nabla \cdot \nabla f = \xi^a(\xi^a f),
\end{equation}
where the sections \( \xi^a \) are viewed as a vector fields on \( G \), hence derivations on $C^\infty(G)$.
In our sign convention, \( \Delta \) has negative eigenvalues.

The Laplace--Beltrami operator is tightly connected to the Casimir invariant \( C := \xi^a \xi^a \in \mathrm{U}(\LieA{g}) \). 
Under the identification of the universal enveloping algebra \( \mathrm{U}(\LieA{g}) \) with the left-invariant differential operators on \( G \), the Casimir invariant maps to the Laplace--Beltrami operator.

Furthermore, consider the tensor product representation
	\begin{equation}
		\rho \otimes \rho : \mathfrak{g} \to \End(V \otimes V), 
		\quad 
		\xi \mapsto \rho(\xi) \otimes \id + \id \otimes \rho(\xi),
	\end{equation}
and the image of the Casimir invariants under $\rho$ and $\rho \otimes \rho$
	\begin{equation}
	\begin{gathered}
	\rho(C) = \rho(\xi^a) \cdot \rho(\xi^a),
	\\
	(\rho \otimes \rho)(C) =
	\left(\rho(\xi^a) \otimes \id + \id \otimes \rho(\xi^a)\right) 
	\cdot 
	\left(\rho(\xi^a) \otimes \id + \id \otimes \rho(\xi^a)\right).
	\end{gathered}
	\end{equation}
We immediately find
	\begin{equation}
		(\rho \otimes \rho)(C) 
		= 
		\rho(C) \otimes \id 
		+ \id \otimes \, \rho(C)
		+ 2 \rho(\xi^a) \otimes \rho(\xi^a).
	\end{equation}
The Casimir invariant $C$ is well known to be independent of the choice of orthonormal basis, hence, by the above equation, so is the following important operator:
	\begin{equation}
	\label{eq:K-Operator}
		K 
		:= 
		\rho(\xi^a) \tensorProd \rho(\xi^a) 
		= 
		\frac{1}{2} \left( (\rho \otimes \rho)(C) 
		-  
		\rho(C) \otimes \id 
		- 
		\id \otimes \, \rho(C) \right) 
		\in 
		\End(V \otimes V).
	\end{equation}
Relative to a basis in \( V \), it assumes the shape
\begin{equation}
	\label{eq:K-OperatorCoords}
  K_{ijkl} = \xi^a_{ij} \xi^a_{kl}.
\end{equation}
	
The operator $K$ defined in \cref{eq:K-Operator} is independent of the choice of the basis \( \xi^a \) (but depends on \( \kappa \)) and it is sometimes called the split Casimir operator.
Indeed, the nondegenerate bilinear form \( \kappa \) on \( \LieA{g} \) yields the isomorphism \( \mathrm{Hom}(\LieA{g}, \LieA{g}) = \LieA{g} \tensorProd \LieA{g}^* \isomorph \LieA{g} \tensorProd \LieA{g} \).
Composing this with the representation \( \rho: \LieA{g} \to \End(V) \) gives a map \( \mathrm{Hom}(\LieA{g}, \LieA{g}) \to \LieA{g} \tensorProd \LieA{g} \to \End(V) \tensorProd \End(V) \).
The image of the identity under this map is  \( K \).
The operator $K$ is related to the image of the Casimir invariant $C$ through contraction:
\begin{equation}
 \label{eq:CasimirAndK-OperatoRelation}
 \rho(C)_{ij} = K_{ikkj}.
\end{equation}
If the representation \( \rho \) is unitary, then for each \( \xi \in \LieA{g} \) the operator \( \rho(\xi) \) is skew-Hermitian and the operator $K$ has the following symmetry properties:
\begin{equation}
	\label{eq:K-OperatorSymmetry}
	K_{ijkl} = \overline{K_{jilk}}, \qquad K_{ikkj} = \overline{K_{jkki}}.
\end{equation}
In particular, $\rho(C)$ is a Hermitian operator. 

Lastly, recall that, by Schur's Lemma, the images of the Casimir invariants under $\rho$ and $\rho \otimes \rho$ are proportional to the identity on irreducible components of $V$ and $V \otimes V$, respectively. By the second representation of $K$ in \cref{eq:K-Operator}, the same holds for $K$.

\subsection{Brownian motion on Lie Groups}
In this section we recall the definition of the Brownian motion on a compact Lie group. The systematic study of this subject goes back to the pioneering work of \parencite{Hunt1956,Yosida1952,Ito1950}, and we refer the reader to \parencites[Section V.35]{RogersWilliams1994}{Liao2004} for textbook treatments.

As before, \( G \) is a compact Lie group whose Lie algebra \( \LieA{g} \) is endowed with an \( \AdAction_G \)-invariant scalar product \( \kappa \), and $\{\xi^a\}$ denotes an orthonormal basis of \( \LieA{g} \).
Let \( (W_t)_{t \geq 0} \) be the unique centered Gaussian process on \( \LieA{g} \) with covariance matrix
\begin{equation}
	\Expect\bigl(W^a_t \, W^b_s\bigr) = \min(t, s) \, \delta^{ab}, \qquad t, s \geq 0,
\end{equation}
where \( W_t^a = \kappa(\xi^a, W_t) \).
The \emph{(Riemannian) Brownian motion} on \( G \) starting at \( g \in G \) is the unique \( G \)-valued stochastic process \( (g_t)_{t \geq 0} \) which solves the Stratonovich stochastic differential equation
\begin{equation}
	\dif g_t = \xi^a(g_t) \difStrat W_t^a, \qquad g_0 = g .
\end{equation}
That is, for every \( f \in C^\infty(G) \),
\begin{equation}
	f(g_t) = f(g) + \int_0^t (\xi^a f) (g_s) \difStrat W_s^a \, .
\end{equation}
Converting into the Itô calculus yields
\begin{equation}  
	f(g_t) = f(g) + \int_0^t (\xi^a f) (g_s) \dif W_s^a + \frac{1}{2} \int_0^t (\Delta f)(g_s) \dif s \, .
\end{equation}
Moreover, \( g_t \) is a Feller diffusion process on \( G \) whose infinitesimal generator, restricted to smooth functions, is one half of the Laplace operator \( \Delta \).

For \( f \in C(G) \) and \( g \in G \), we denote by \( \Expect_g\bigl(f(g_t)\bigr) = \Expect\bigl(f(g_t) | g_0 = g\bigr) \) the conditional expectation of \( f \) given that \( g_t \) starts at \( g \).
The resulting semigroup is a contraction on \( C(G) \) and satisfies
\begin{equation}
	\Expect_g\bigl(f(g_t)\bigr) = \int_G f(a) \, p_t (g^{-1}a) \, \dif a,
\end{equation}
where, for \( t > 0 \), \( p_t: G \to \R \) is the smooth probability density satisfying the heat equation
\begin{equation}
	\label{eq:brownianMotion:heatEquation}
	\frac{1}{2}\Delta p_t = \frac{\partial}{\partial t} p_t, \qquad \lim_{t \to 0} p_t = \delta_e.
\end{equation}
Usually, we are interested only in processes starting at the identity and then abbreviate \( \Expect \equiv \Expect_e \).

\subsection{Wilson Loops}

In this section, we recall basic elements of the lattice gauge theory.
The reader is referred to the textbooks \parencite{RudolphSchmidt2014,MontvayMuenster1994} for a detailed treatment.

Consider a directed graph \( (\Lambda^0, \Lambda^1) \), which one may think of as being embedded either in space or spacetime.
Here, \( \Lambda^0 \) is the set of all vertices which we assume to be finite, and \( \Lambda^1_+ \) is the set of all directed edges (\ie, ordered pairs of vertices).
For an edge \( e \), let \( s(e), t(e) \in \Lambda^0 \) be its source and target vertices, respectively, and let \( e^{-1} \) be the edge going in the opposite direction.
We denote by \( \Lambda^1_- = \set{e^{-1}: e \in \Lambda^1_+} \) the set of all edges with their orientation reversed, and set \( \Lambda^1_\pm = \Lambda_+ \cup \Lambda_- \).
A (field) configuration is a map \( g: \Lambda^1_+ \to G \) assigning to each edge \( e \) a group element \( g(e) \) that should be thought of as the approximation of the parallel transport along that edge.
We extend \( g \) to a map \( g: \Lambda^1_\pm \to G \) by setting \( g(e^{-1}) = g(e)^{-1} \).

An (oriented) path \( l = (e_1, \dotsc, e_r) \) is an ordered tuple of edges \( e_i \in \Lambda^1_\pm \) such that \( t(e_i) = s(e_{i+1}) \) for all \( 1 \leq i < r \).
A path is called a loop if the edges form a cycle, \ie \( t(e_r) = s(e_1) \).
If, additionally, each edge occurs only once then the loop is called a plaquette (or face).
Given a path \( l \), the product of a configuration \( g \) along \( l \) is defined by \( g(l) = g(e_1) \dotsm g(e_r) \).
Given a choice of a set \( \Lambda^2 \) of plaquettes, the probability density on the space of configurations is given by the Wilson action
\begin{equation}
	\label{eq:lattice:WilsonAction}
	g \mapsto \frac{1}{Z} \exp\left(\beta \sum_{p \in \Lambda^2} \tr_\rho \bigl(g(p)\bigr) \right),
\end{equation}
where \( Z \) is a normalization factor (the partition function) and \( \beta \in \R \) is the so-called inverse temperature.
Here, the trace is taken with respect to a representation \( \rho \) of \( G \), which usually is assumed to be irreducible or even to be the fundamental representation.
One is mainly interested in expectation values of Wilson loop observables.
These are functions \( W_l \) on the space of configurations indexed by loops \( l = (e_1, \dotsc, e_r) \) and are given by
\begin{equation}
	W_l(g) = \tr_\rho \bigl(g(l)\bigr) = \tr_\rho \bigl(g(e_1)^{\pm 1} \dotsm g(e_r)^{\pm 1}\bigr) \,,
\end{equation}
where the sign in the factor \( g(e_i)^{\pm 1} \) is determined based on whether \( e_i \) is an element of \( \Lambda^1_+ \) or \( \Lambda^1_- \).
In other words, one is lead to calculate integrals of the form
\begin{equation}
	\int W_l(g) \, \exp\left(\beta \sum_{p \in \Lambda^2} \tr_\rho \bigl(g(p)\bigr) \right) \dif g \, .
\end{equation}
Since \( \dif g = \prod_{e \in \Lambda^1_+} \dif g_e \) is the product of Haar measures, one can evaluate such an integral by successively integrating over copies of \( G \).

In the following, we are mainly concerned with the resulting integral over a single edge.
For this, it is convenient to change the notation and language slightly and consider restrictions to a single edge.
For the lattice gauge theory calculations that one may want to perform in the end, it is good to remember that Wilson loops do, in fact, depend on many copies of $G$.
\begin{definition}[Wilson Loops]
	Let \( \rho: G \to \End(V) \) be a finite-dimensional representation of \( G \), and for some natural number $r \in \mathbb{N}$, fix an element 
	\begin{equation}
	l = ((c_1,\sigma_1), \dots (c_r, \sigma_r)) \in (G \times \{\pm 1\})^r.
	\end{equation}
	The \emph{(single-argument) Wilson loop} $W_{\rho, l}$ associated with this data is given by
	\begin{equation}
		\label{eq:GeneralRepresentationOfWilsonLoop}
		W_{\rho, l}: G \to \mathbb{C}, \qquad W_{\rho, l}(g) = \tr_\rho( c_1 g^{\sigma_1} c_2 g^{\pm1} \dotsm c_r g^{\sigma_r}).
	\end{equation}
	Often the representation is clear from the context and we simply write \( W_l \) in this case.
	Moreover, we say that the $g^{\pm1}$ between $c_s$ and $c_{s+1}$ is \emph{in the $s$-th position}.
\end{definition}
In this setting, we consider the probability measure to be a single-argument version of the \emph{Wilson action}:
\begin{equation}
	\label{eq:WilsonAction}
	\nu_W(g) = \frac{1}{Z} \exp\left(\beta \sum_{p} W_p (g) \right), \quad g \in G,
\end{equation}
where \( Z \) is a suitable normalization factor, $\beta \in \mathbb{R}$ a fixed number, the sum is over a finite set that is not further specified, and $W_p$ is a single-argument Wilson loop of the form
\begin{equation}
W_p(g) = \tr_\rho(C_p g^{\pm 1}), \quad C_p \in G.
\end{equation}

Below, we also need a slight generalization of a single-argument Wilson loop for which the coefficients are not necessarily elements of the same group.
\begin{definition}[Generalized Wilson Loops]
\label{def:GeneralizedWilsonLoops}
	Let \( \rho: G \to \End(V) \) be a finite-dimensional representation of \( G \), and for some natural number $r \in \mathbb{N}$, fix an element 
	\begin{equation}
	l = ((c_1,\sigma_1), \dots (c_r, \sigma_r)) \in (\End(V) \times \{\pm 1\})^r.
	\end{equation} 
	The \emph{generalized Wilson loop \( W_{\rho, l} \)} associated with this data is given by
	\begin{equation}
		W_{\rho, l}: G \to \C, \qquad W_l(g) = \tr_{V}\bigl(c_1 \rho(g^{\sigma_1}) c_2 \rho(g^{\sigma_2}) \dotsm c_r \rho(g^{\sigma_r})\bigr).
	\end{equation}

	A (generalized) Wilson loop is called \emph{linear} if there is only one factor of \( g \) in the above representation, \ie \( r = 1 \); otherwise it is called \emph{polynomial}.
\end{definition}
The notion of a generalized Wilson loop is inspired by the concept of spin networks.
In fact, a generalized Wilson loop can be visualized as a loop in a graph consisting of two vertices and \( n + 1 \) directed edges, where one edge is decorated by \( \rho(g) \) and the other edges by the endomorphisms \( c_i \).

A linear generalized Wilson loop (with, say, positive exponent on the \( g \) factor) is completely determined by its coefficient \( c \in \End(V) \).
We thus obtain a map
\begin{equation}
	\End(V) \to C^0(G, \C), \qquad c \mapsto W_{\rho, (c)} = \tr_{V}\bigl(c \, \rho(\cdot)\bigr).
\end{equation}
Under the isomorphism \( \End(V) \cong V^* \tensorProd V \) this is nothing but the usual embedding of matrix coefficients.
In other words, linear generalized Wilson loops are just linear combinations of matrix coefficients, and every matrix coefficient is a linear generalized Wilson loop.
\begin{proposition}
	\label{prop:wilsonLoops:expandInLinearLoops}
	Every (generalized) Wilson loop \( W_{\rho, l} \) can be written as a finite linear combination of linear generalized Wilson loops associated with irreducible representations.
	That is, there exists a finite set of irreducible representations \( (\tau, V_\tau) \) of \( G \) and a collection of endomorphisms \( c_\tau \in \End(V_\tau) \) such that
	\begin{equation}
		W_{\rho, l} = \sum_{\tau} W_{\tau, (c_\tau)}.
	\end{equation}
\end{proposition}
\begin{proof}
	Note that a (generalized) Wilson loop transforms under left translation as \( W_{\rho, l}(a g) = W_{\rho, a \cdot l}(g) \) where the action \( a \cdot l \) of \( a \in G \) on the coefficients \( c_i \) is either by left or right translation or conjugation depending on the signatures.
	This shows that the linear span \( \mathrm{span} \, W_{\rho, \cdot} \) of all generalized Wilson loops (relative to a given representation \( \rho \)) is a left \( G \)-translation invariant subspace of \( C^0(G, \C) \).
	By choosing a basis in \( \End(V) \), we obtain a finite spanning set so that \( \mathrm{span} \, W_{\rho, \cdot} \) is finite-dimensional.
	Hence, every (generalized) Wilson loop is a so-called \emph{representative function}, see \parencite[Definition~III.1.1]{BrockerDieck1985}.
	By \parencite[Proposition~III.1.5]{BrockerDieck1985}, every representative function is a finite linear combination of matrix coefficients with respect to irreducible representations.
	As we have remarked above, the latter are linear generalized Wilson loops.
\end{proof}

\section{Integration by parts and expectation of Wilson loops}
\label{sec:expectationWilsonLoops}

The following identity is of fundamental importance for us, and it is derived by a simple application of integration by parts.
\begin{lemma}
	\label{prop:IntegrationByParts}
	Let \( G \) be a compact Lie group and let \( \nu \) be a probability density with respect to the normalized Haar measure on \( G \).
	For smooth functions $F_{1}, \dots, F_{q}$ on $G$, 
	\begin{equation}\begin{split}
	\sum_{r=1}^q \int_G (\Delta F_{r}) \, & F_{1} \dotsm \widehat{F_{r}} \dotsm F_{q} \, \nu \, \dif g
	=
	\int_G F_{1} \dotsm F_{q} \, \Delta \nu \, \dif g
	\\
	&- 2 \sum_{\substack{r, s = 1 \\ r < s}}^q \int_G \scalarProd{\dif F_{r}}{\dif F_{s}} \,  F_{1} \dotsm \widehat{F_{r}} \dotsm \widehat{F_{s}} \dotsm F_{q} \, \nu \, \dif g 
	\end{split}\end{equation}
	where the hat signifies omission of the corresponding term.
\end{lemma}
\begin{proof}
	Using integration by parts twice, we obtain
	\begin{equation}\begin{split}
		- \sum_r \int_G &(\Delta F_{r}) \, F_{1} \dotsm \widehat{F_{r}} \dotsm F_{q} \, \nu \, \dif g
		\\
		&=
		\sum_r \int_G \scalarProd{\dif F_{r}}{\dif \nu} \, F_{1} \dotsm \widehat{F_{r}} \dotsm F_{q} \, \dif g
		\\
		&\qquad+ \sum_{r \neq s} \int_G \scalarProd{\dif F_{r}}{\dif F_{s}} \,  F_{1} \dotsm \widehat{F_{r}} \dotsm \widehat{F_{s}} \dotsm F_{q} \, \nu \, \dif g
		\\
		&=
		- \sum_r \int_G \Delta \nu \, F_{1} \dotsm F_{q} \, \dif g
		\\
		&\qquad+ 2 \sum_{r < s} \int_G \scalarProd{\dif F_{r}}{\dif F_{s}} \,  F_{1} \dotsm \widehat{F_{r}} \dotsm \widehat{F_{s}} \dotsm F_{q} \, \nu \, \dif g
	\end{split}\end{equation} 
	and the claimed equality follows immediately.
\end{proof}
In this section, we will make use of this basic lemma by applying it to a family $W_{l_1}, \dots, W_{l_q}$ of single-argument Wilson loops.
For simplicity, we consider only the case where all Wilson loops are defined with respect to the same representation \( \rho \) and where the coefficients are elements of the group (\ie, normal Wilson loops instead of generalized ones).
However, with minor modifications, everything we say generalizes to generalized Wilson loops with respect to possibly different representations, see \cref{rem:theoremA:forGeneralizedWilsonLoops} below for more details.
The significance of \cref{prop:IntegrationByParts} lies in the fact that, for Wilson loops, both sides of the relation can be evaluated and this yields a non-trivial identity.
Moreover, each side has a geometric interpretation: The Laplacian \( \Delta W_l \) of a Wilson loop gives rise to what we call the twisting of \( W_l \), and the inner product \( \scalarProd{\dif W_{l}}{\dif W_{l'}} \) of two Wilson loops yields their merging.

\subsection{Merging: Calculation of the right-hand side}
Within this subsection, we will focus on the calculation of the term involving the inner product of two Wilson loops.
Given a Wilson loop $W_l$, let $E_+(l)$ be the positions $j$ where $W_l$ has the identity $g$ and $E_-(l)$ the positions where $W_l$ has the inverse $g^{-1}$. Let $E(l) := E_+(l) \cup E_-(l)$. Similarly for $E_+{(l')}$,$E_-{(l')}$ and $E{(l')}$.  
Consider the matrix component functions $g_{ij} := \rho(g)_{ij} : G \to \mathbb{C}$ and $\xi_{ij} := \rho(\xi)_{ij} : \mathfrak{g} \to \mathbb{C}$. We find:
\begin{equation}
\label{eq:MatrixElementsDifferential}
d g_{ij} = g_{il} \cdot \xi^a_{lj} \cdot  \epsilon^a,
\quad
d (g^{-1})_{ij} = - \xi^a_{il} \cdot (g^{-1})_{lj} \cdot \epsilon^a.
\end{equation}
Using these expressions for the differentials $dg_{ij}$ and $d g_{ij}^{-1}$, and the expression~\eqref{eq:GeneralRepresentationOfWilsonLoop} for a general single-argument Wilson loop, we find
\begin{equation}
\label{eq:ExpansionOfdWl}
\begin{aligned}
&d W_l = \\
&\sum_{j \in E_+(l)} 
(c_1 g^{\pm 1} \dotsm c_{j-1} g^{\pm 1} c_j g)^{}_{k_1 k_2}
\xi^a_{k_2 k_3} 
(c_{j+1} g^{\pm 1} \dotsm c_n g^{\pm 1})^{}_{k_3 k_1} \, d\xi^a\\
&-
\sum_{j \in E_-(l)} 
(c_1 g^{\pm 1} \dotsm c_{j-1} g^{\pm 1} c_j)^{}_{k_1 k_2} 
\xi^a_{k_2 k_3} 
(g^{-1} c_{j+1} g^{\pm 1} \dotsm c_n g^{\pm 1})^{}_{k_3 k_1} \, d\xi^a. 
\end{aligned}
\end{equation}
To simplify notation, we introduce the following definition.
\begin{definition}[Merging loops in general representations] $ $ \newline
\label{def:MergingLoopsGenerally}
For two single-argument Wilson loops of the form $W_l(g) = \tr_\rho(C g^{\sigma_1}), W_{l'}(g) = \tr_\rho(D g^{\sigma_2} )$ with $C,D \in G$ and exponents $\sigma_1, \sigma_2 \in \{\pm 1\}$, we define their \emph{merging} $\mathcal{M}(W_l, W_{l'}) : G \to \mathbb{C}$, depending on the value of the tuple of exponents $(\sigma_1, \sigma_2)$, as follows:
 \begin{equation}
  \mathcal{M}(W_l, W_{l'})(g) =
  \begin{cases}
   +\tr_\rho( Cg \xi^a) \cdot \tr_\rho( D  g \xi^a)
    &\text{ if } (\sigma_1, \sigma_2) = (+,+),\\
   -\tr_\rho( C g \xi^a )  \cdot \tr_\rho( D  \xi^a g^{-1})
	& \text{ if } (\sigma_1, \sigma_2) = (+,-), \\
   - \tr_\rho( C \xi^a g^{-1})  \cdot \tr_\rho( D g \xi^a )
   	& \text{ if } (\sigma_1, \sigma_2) = (-,+), \\
   +\tr_\rho( C \xi^a g^{-1})  \cdot \tr_\rho( D  \xi^a g^{-1}) 	& \text{ if } (\sigma_1, \sigma_2) = (-,-).
  \end{cases}	  
 \end{equation}
Note the implicit sum over the Lie algebra index $a$ in all of the above.
The merging of two \emph{generalized} Wilson loops is defined analogously.
Since the above case distinctions depending on the tuple $(\sigma_1, \sigma_2)$ will occur more often later, we will adopt the following equivalent notation for brevity:
 \begin{equation}
  \mathcal{M}(W_l, W_{l'})(g) =
  \begin{cases}
  (+,+): &+\tr_\rho( Cg \xi^a) \cdot \tr_\rho( D  g \xi^a),\\
  (+,-): &-\tr_\rho( C g \xi^a )  \cdot \tr_\rho( D  \xi^a g^{-1})\\
   (-,+): &- \tr_\rho( C \xi^a g^{-1})  \cdot \tr_\rho( D g \xi^a ), \\
   (-,-): &+\tr_\rho( C \xi^a g^{-1})  \cdot \tr_\rho( D  \xi^a g^{-1}).
  \end{cases}	  
 \end{equation}
For two arbitrary single-argument Wilson loops \( W_l \) and \( W_{l'} \) with distinguished factors $g^{\pm 1}$ in the, respectively, $j$-th and $j'$-th position, their \emph{merging \( \mathcal{M}_{jj'}(W_l, W_{l'}) \) at the $j$-th and $j'$-th positions} is defined by the same formulas after the Wilson loops have been expressed in the above form\footnotemark{} with \( C, D \) possibly depending on \( g \).\footnotetext{That is, using the cyclicity of the trace: \( W_l(g) = \tr_\rho( c_1 g^{\pm1} c_2 g^{\pm1} \dotsm c_n g^{\pm1}) = \tr_\rho (c_{j+1}g^{\pm1} c_{j+2}g^{\pm1} \dotsm c_n g^{\pm1} c_1 g^{\pm1} c_2 g^{\pm1} \dotsm c_j g^{\pm1}) \).}
The \emph{total merger} of two loops $W_l,W_{l'}$ is defined as
	\begin{equation}
		\mathcal{M}(W_{l}, W_{l'}) :=
		\sum_{\substack{j \in E(l), \\ j' \in E(l')}}  
		\mathcal{M}_{jj'}(W_l, W_{l'}).
	\end{equation}
\end{definition}
\begin{remark}
	Note that this is \emph{not} equal to what in \cite{Chatterjee2019} is called the merging of loops when $G = SO(N)$. However, there is a relation between the notions, which is outlined in \cref{sec:examples:so}.
\end{remark}
The particular form of the merge operation depends on the Lie algebra under study, and it is completely controlled by the operator $K$ \cref{eq:K-OperatorCoords}. In fact, we have
\begin{equation}
	\label{eq:MergingRulesGenerally}
	\mathcal{M}(W_l, W_{l'})(g)
	= K_{ijkl} \cdot \begin{cases}
	(+,+):  \quad + \, C_{js} g_{si} \; \, D_{lt} g_{tk}, \\
	(+,-):  \quad - \, C^{}_{js} g^{}_{si} \; \, D^{}_{tk} g^{-1}_{lt}, \\
	(-,+):  \quad - \, C^{}_{si} g^{-1}_{js} \, D^{}_{lt} g^{}_{tk}, \\
	(-,-):  \quad + \, C^{}_{si} g^{-1}_{js} \, D^{}_{tk} g^{-1}_{lt}. \\
	\end{cases}
\end{equation}
Identities expressing \( K \) in terms of elementary matrices are called \emph{completeness relations}.
These relations usually allow one to rewrite \( \mathcal{M}_{jj'}(W_l, W_{l'}) \) as a linear combination of certain Wilson loops.
Below in \cref{sec:examples} we discuss this exemplarily for \( G = \OGroup(N) \), \(\SpGroup(N) \), \( \UGroup(N) \), \( \SUGroup(N) \) in more detail.
Note, however, that in general the merging of two Wilson loops is not a linear combination of Wilson loops again as the example of \( G_2 \) shows.
On the other hand, the class of \emph{generalized} Wilson loops is closed under the merging operation.
\begin{proposition}
	\label{prop:merging:againWilsonLoop}
	The merging of two generalized Wilson loops \( W_{C_1, \theta_1}(g) = \tr_V\bigl({C_1} \rho(g^{\theta_{1}})\bigr) \) and \( W_{C_2, \theta_2}(g) = \tr_V\bigl({C_2} \rho(g^{\theta_{2}})\bigr) \) is the generalized Wilson loop
	\begin{equation}
		\mathcal{M}(W_{C_1, \theta_1}, W_{C_2, \theta_2})(g) = \tr_{V^{\theta_1, \theta_2}}\bigl(K^{\theta_1, \theta_2} \, \psi^{\theta_1,\theta_2}(C_1, C_2) \, \rho^{\theta_1, \theta_2}(g)\bigr),
	\end{equation}
	where, depending on the signatures \( (\theta_1, \theta_2) \), the representation \( \rho^{\theta_1, \theta_2} \) is defined by
	\begin{equation}
		V^{\theta_1, \theta_2} = \begin{cases}
			(+,+):  \quad V \tensorProd V, \\
			(+,-):  \quad V \tensorProd V^*, \\
			(-,+):  \quad V^* \tensorProd V, \\
			(-,-):  \quad V^* \tensorProd V^*, \\
			\end{cases}
		\qquad
		\rho^{\theta_1, \theta_2} = \begin{cases}
			(+,+):  \quad \rho \tensorProd \rho, \\
			(+,-):  \quad \rho \tensorProd \rho^*, \\
			(-,+):  \quad \rho^* \tensorProd \rho, \\
			(-,-):  \quad \rho^* \tensorProd \rho^*, \\
			\end{cases}
	\end{equation}
	and the map \( \psi^{\theta_1,\theta_2}: \End(V)^2 \to \End(V^{\theta_1, \theta_2}) \) is defined by
	\begin{equation}
		\psi^{\theta_1,\theta_2}(C, D) = \begin{cases}
			(+,+):  \quad C \tensorProd D, \\
			(+,-):  \quad C \tensorProd D^*, \\
			(-,+):  \quad C^* \tensorProd D, \\
			(-,-):  \quad C^* \tensorProd D^*, \\
			\end{cases}
	\end{equation}
	and \( K^{\theta_1, \theta_2} = \rho^{\theta_1}(\xi^a) \tensorProd \rho^{\theta_2}(\xi^a) \) with \( \rho^{+} = \rho \) and \( \rho^{-} = \rho^* \).
\end{proposition}
\begin{proof}
	For simplicity, we only give the proof for the case \( (\theta_1, \theta_2) = (+,-) \); the other cases are analogous.
	Since the trace is invariant under transposition, we have
	\begin{equation}
		\tr_V \bigl(C \rho(\xi) \rho(g^{-1})\bigr) = \tr_{V^*} \bigl(\rho(g^{-1})^* \rho(\xi)^* C^*\bigr) = - \tr_{V^*} \bigl(\rho^*(g) \rho^*(\xi) C^*\bigr)
	\end{equation}
	for \( C \in \End(V) \), \( \xi \in \LieA{g} \) and \( g \in G \).
	Thus, by~\eqref{def:MergingLoopsGenerally}, we find
	\begin{equation}\begin{split}
		\mathcal{M}(W_{C_1, \theta_1}, W_{C_2, \theta_2})(g)
			&= - \tr_{V}\bigl(C_1 \rho(g) \rho(\xi^a)\bigr) \cdot \tr_{V}\bigl(C_2 \rho(\xi^a) \rho(g^{-1})\bigr)
			\\
			&= \tr_{V}\bigl(\rho(\xi^a) C_1 \rho(g)\bigr) \cdot \tr_{V^*}\bigl(\rho^*(\xi^a) C_2^* \rho^*(g)\bigr)
			\\
			&= \tr_{V \tensorProd V^*}\bigl(\rho(\xi^a) \tensorProd \rho^*(\xi^a) \circ C_1 \tensorProd C_2^* \circ \rho(g) \tensorProd \rho^*(g) \bigr),
	\end{split}\end{equation}
	which finishes the proof.
\end{proof}

With the above notation and Equation~\eqref{eq:ExpansionOfdWl}, we arrive at the following expression:
	\begin{equation}
		\label{eq:MergingLoopsGenerally}
		\langle d W_l, d W_{l'} \rangle =
		\sum_{\substack{j \in E(l), \\ j' \in E(l')}}  
		\mathcal{M}_{jj'}(W_l, W_{l'})
		= \mathcal{M}(W_{l}, W_{l'})
	\end{equation}
\subsection{Twisting: Calculation of the left-hand side}
Let us now examine what the action of the Laplacian on a Wilson loop. We start off by using the higher-order product rule
	\begin{equation}
		\Delta(f \cdot h) = f \cdot \Delta h + \Delta(f) \cdot h + 2 \, \langle \dif f, \dif h \rangle
	\end{equation}
	for \( f,h \in C^\infty(G) \).	
In the representation~\eqref{eq:GeneralRepresentationOfWilsonLoop} for a single-argument Wilson loop $W_l$, this yields:
	\begin{equation}
	\begin{aligned}
		&\Delta W_l =
		\sum_{j=1}^n (c_1 g^{\pm 1} \dotsm c_j)_{i_1 i_2}  
		\Delta(g^{\pm 1}_{i_2 i_3}) 
		(c_{j+1} \dotsm c_n g^{\pm 1})_{i_3 i_1}
		\\
		&+
		\sum_{j \neq k} 
		\langle d g^{\pm 1}_{i_{2} i_3}, d g^{\pm 1}_{i_4 i_5} \rangle
		(c_1 g^{\pm1} \dotsm c_j)_{i_1 i_2} 
		(c_{j+1} \dotsm c_k)_{i_{3} i_{4}} 
		(c_{k+1} \dotsm c_n g^{\pm1})_{i_5 i_1}.
		\end{aligned}
	\end{equation}
 Recall that by the Peter-Weyl theorem, matrix elements of irreducible representations are eigenfunctions of the Laplacian, and matrix elements to the same irreducible representation lie in the same eigenspace. Hence, the first sum in the above is a scalar multiple of \( W_l \). 
The mixed term takes a form that is very similar to the mergers of \cref{def:MergingLoopsGenerally}, except that the loop is ``merged with \emph{itself}'', in two different locations. Let us make this precise with the following definition.
\begin{definition}[Twisting loops in general representations]
Given a single-ar\-gu\-ment loop $W_l(g) = \tr_\rho(Cg^{\sigma_1} D g^{\sigma_2})$ with $C,D \in G$ and exponents $\sigma_1, \sigma_2 \in \{\pm 1\}$.
We define its \emph{twisting} $\mathcal{T}(W_l) : G \to \mathbb{C}$, depending on the value of the tuple of exponents $(\sigma_1, \sigma_2)$, as follows:
	\begin{equation}
		\mathcal{T}(W_l)(g) =
			\begin{cases}
			(+,+): &+\tr_\rho(C g \xi^a D g \xi^a), \\
			(+,-): &-\tr_\rho(C g \xi^a D \xi^a g^{-1}),  \\
			(-,+): &-\tr_\rho(C \xi^a g^{-1} D g \xi^a),   \\
			(-,-): &+\tr_\rho(C \xi^a g^{-1} D \xi^a g^{-1}).
			\end{cases}		
	\end{equation}
For an arbitrary single-argument Wilson loops \( W_l \) with distinguished factors $g^{\pm 1}$ in the, respectively, $j$-th and $j'$-th position, its \emph{twisting \( \mathcal{T}_{jj'}(W_l) \) at the $j$-th and $j'$-th positions} is defined by the same formulas after the Wilson loops have been expressed in the above form with \( C, D \) possibly depending on \( g \) (cf. \cref{def:MergingLoopsGenerally}).
The \emph{total twisting} of a loop $W_l$ is defined as
	 \begin{equation}
		 \mathcal{T}(W_l) :=
		 \sum_{\substack{j, j' \in E(l) \\ j \neq j' }} 
		 \mathcal{T}_{jj'}(W_l).
	 \end{equation}
\end{definition}
Note that the twisting of a Wilson loop, too, is completely determined by the operator \( K \):
\begin{equation}
	\label{eq:TwistingRulesGenerally}
	\mathcal{T}(W_l)(g)
	= K_{ijkl} \cdot \begin{cases}
	(+,+):  \quad + \, C_{ls} g_{si} \; \, D_{jt} g_{tk}, \\
	(+,-):  \quad - \, C^{}_{ts} g^{}_{si} \; \, D^{}_{jk} g^{-1}_{lt}, \\
	(-,+):  \quad - \, C^{}_{li} g^{-1}_{js} \, D^{}_{st} g^{}_{tk}, \\
	(-,-):  \quad + \, C^{}_{ti} g^{-1}_{js} \, D^{}_{sk} g^{-1}_{lt}. \\
	\end{cases}
\end{equation}
This formula should be compared with the expression~\eqref{eq:MergingRulesGenerally} for the merging, which has the same structure in \( K \) and \( \rho(g^{\pm 1}) \tensorProd \rho(g^{\pm1}) \) but the contraction with the tensor \( \rho(C) \tensorProd \rho(D) \) is different.
As a consequence of Schur's lemma, the matrix elements are eigenfunctions of the Laplace operator $\Delta$.
Using its definition~\eqref{eq:DefinitionLaplace} and \cref{eq:CasimirAndK-OperatoRelation}, we find
\begin{equation}
	\label{eq:EigenvalueK-OperatorGeneral:calculation}
	\lambda g_{ij} \stackrel{!}{=} \Delta g_{ij} = \xi^a (\xi^a g_{ij}) = g^{}_{il} \xi^a_{lk} \xi^a_{kj} = g_{il} K_{lkkj} = g_{il} \rho(C)_{lj}.
\end{equation}
Hence the eigenvalue \( \lambda \) of the Laplace operator equals the eigenvalue of the Casimir invariant $C \in \mathrm{U}(\LieA{g})$ in the representation $\rho$:
\begin{equation}
	\label{eq:EigenvalueK-OperatorGeneral}
	\rho(C)_{ij} = K_{ikkj} = \lambda \delta_{ij}.
\end{equation}
Thus we can rewrite the Laplacian of a Wilson loop $W_l$ in terms of the twisting $\mathcal{T}( W_l )$, the eigenvalue $\lambda$, and the number $n$ counting the amount of $g^{\pm 1}$-factors contained in $W_l$:
	 \begin{equation}
		\Delta W_l =
		\lambda \cdot n \cdot W_l
		+
		\sum_{j \neq j' \in E(l)} \mathcal{T}_{jj'}(W_l) = \lambda \cdot n \cdot W_l + \mathcal{T}(W_l).
	 \end{equation}
\subsection{Synthesis}
Combining the calculated terms with \cref{prop:IntegrationByParts} we get the following theorem.
\begin{theorem}
\label{prop:theoremA}
Let $G$ be a compact Lie group equipped with a probability density \( \nu \), $\rho: G \to V$ an irreducible, finite-dimensional representation and $W_{l_1}, \dots , W_{l_q}\colon G \to \mathbb{C}$ a collection of single-argument Wilson loops.
Let $\lambda \in \mathbb{C}$ be the eigenvalue of the Casimir \( \rho(C) \), and denote the number of factors of \( g \) or \( g^{-1} \) in the canonical representation of the Wilson loop \( W_{l_r} \) by \( n_r \).
Then we have
	\begin{equation}
		\label{eq:theoremA}
		\begin{aligned}
			\lambda \sum_{r=1}^q  n_r \cdot \int_G W_{l_1} &\dotsm W_{l_q}  \, \nu \, \dif g
			=
			\\
			&- 2 \sum_{\substack{r, s = 1 \\ r < s}}^q  \int_G  \mathcal{M}(W_{l_r},W_{l_s}) \cdot W_{l_1} \dotsm \widehat{W}_{l_r} \dotsm \widehat{W}_{l_s} \dotsm W_{l_q}  \, \nu \,  \dif g
			\\
			&-\sum_{r=1}^q \int_G \mathcal{T}(W_{l_r}) \cdot  W_{l_1} \dotsm \widehat{W}_{l_r} \dotsm W_{l_q}  \, \nu \, \dif g
			\\
			&+ \int_G W_{l_1} \dotsm W_{l_q} \, \Delta \nu \, \dif g.
		\end{aligned}
	\end{equation}
\end{theorem}

As applications, let us state \cref{prop:theoremA} for the three different choices of probability densities \( \nu \) introduced above.
\begin{corollary}[Haar measure]
	\label{prop:theoremA:haar}
	In the setting of \cref{prop:theoremA}, we have
	\begin{equation}
		\begin{aligned}
			\lambda \sum_{r=1}^q  n_r \cdot \int_G W_{l_1} &\dotsm W_{l_q}  \, \dif g
			=
			\\
			&- 2 \sum_{\substack{r, s = 1 \\ r < s}}^q  \int_G  \mathcal{M}(W_{l_r},W_{l_s}) \cdot W_{l_1} \dotsm \widehat{W}_{l_r} \dotsm \widehat{W}_{l_s} \dotsm W_{l_q}  \,  \dif g
			\\
			&-\sum_{r=1}^q \int_G \mathcal{T}(W_{l_r}) \cdot  W_{l_1} \dotsm \widehat{W}_{l_r} \dotsm W_{l_q}  \, \dif g .
		\end{aligned}
	\end{equation}
\end{corollary}

\begin{corollary}[Brownian motion]
	\label{prop:theoremA:brownian}
	In the setting of \cref{prop:theoremA}, we have
	\begin{equation}
		\begin{aligned}
			\int_G W_{l_1} &\dotsm W_{l_q}  \, p_t \, \dif g
			= W_{l_1}(e) \dotsm W_{l_q}(e)
			\\
			&\qquad+ \frac{1}{2} e^{\frac{t}{2} \lambda \sum_{r=1}^q  n_r}
			\int_0^t
			\mathcal{MT}(s) 
				\, e^{- \frac{s}{2} \lambda \sum_{r=1}^q  n_r} \, \dif s,
		\end{aligned}
	\end{equation}
	where
	\begin{equation}\begin{split}
		\mathcal{MT}(t) &=
		2 \sum_{\substack{r, s = 1 \\ r < s}}^q \int_G \mathcal{M}(W_{l_r},W_{l_s}) \cdot W_{l_1} \dotsm \widehat{W}_{l_r} \dotsm \widehat{W}_{l_s} \dotsm W_{l_q} \, p_t \, \dif g
		\\
		&\quad+ \sum_{r=1}^q \int_G \mathcal{T}(W_{l_r}) \cdot  W_{l_1} \dotsm \widehat{W}_{l_r} \dotsm W_{l_q} \, p_t \, \dif g.
	\end{split}\end{equation}
\end{corollary}
\begin{proof}
	Using the heat equation,~\eqref{eq:theoremA} reduces, for the Brownian motion, to a first-order linear differential equation of the form
	\begin{equation}
		c f(t) - 2 f'(t) = h(t),
	\end{equation}
	where \( f(t) \) is the expectation of the product of Wilson loops, and \( h(t) \) includes the merging or twisting terms.
	This equation has the general solution
	\begin{equation}
		f(t) = f(0) - \frac{1}{2} e^{\frac{c}{2} t} \int_0^t h(s) e^{- \frac{c}{2} s} \, \dif s.
	\end{equation}
	Since the heat kernel approaches the delta distribution at the identity as \( t \to 0 \), the initial value is \( f(0) = W_{l_1}(e) \dotsm W_{l_q}(e) \).
	This completes the proof.
\end{proof}

\begin{corollary}[Wilson action]
	\label{prop:theoremA:wilson}
	In the setting of \cref{prop:theoremA}, we have
	\begin{equation}
		\label{eq:theoremA:wilson}
		\begin{aligned}
			\lambda \sum_{r=1}^q  n_r \cdot \int_G &W_{l_1} \dotsm W_{l_q}  \, \nu_W \, \dif g
			=
			\\
			&- 2 \sum_{\substack{r, s = 1 \\ r < s}}^q  \int_G  \mathcal{M}(W_{l_r},W_{l_s}) \cdot W_{l_1} \dotsm \widehat{W}_{l_r} \dotsm \widehat{W}_{l_s} \dotsm W_{l_q}  \, \nu_W \,  \dif g
			\\
			&-\sum_{r=1}^q \int_G \mathcal{T}(W_{l_r}) \cdot  W_{l_1} \dotsm \widehat{W}_{l_r} \dotsm W_{l_q}  \, \nu_W \, \dif g
			\\
			&+\beta \lambda \sum_{p} \int_G W_p  \cdot W_{l_1} \dotsm W_{l_q} \,  \nu_W \, \dif g
			\\
			&+ \beta^2 \sum_{p,p'} \int_G \mathcal{M}(W_p, W_{p'}) \cdot W_{l_1} \dotsm W_{l_q} \, \nu_W \, \dif g.
		\end{aligned}
	\end{equation}
\end{corollary}
\begin{proof}
	The general identity \( \Delta \exp(f) = \bigl(\Delta f + \dualPair{\dif f}{\dif f}\bigr) \exp(f) \) implies for the Wilson action defined in~\eqref{eq:WilsonAction} that
	\begin{equation}\begin{split}
		\Delta \nu_W
			&= \Bigl(\beta \sum_{p} \Delta W_p + \beta^2 \sum_{p,p'} \dualPair{\dif W_p}{\dif W_{p'}}\Bigr) \, \nu_W
			\\
			&= \Bigl(\beta \lambda \sum_{p} W_p + \beta^2 \sum_{p,p'} \mathcal{M}(W_p, W_{p'})\Bigr) \, \nu_W,
	\end{split}\end{equation}
	where, in the second line, we used that Wilson loops with a single argument of \( g \) are eigenfunctions of the Laplacian and that the scalar product of two such loops equals their merging.
	Inserting this equality in~\eqref{eq:theoremA} yields~\eqref{eq:theoremA:wilson}. 
\end{proof}

\Cref{prop:theoremA:wilson} is essentially a generalization of \parencite[Thm 8.1]{Chatterjee2019}, which studies the case $G = \SOGroup(N)$ in the fundamental representation.
The main difference between our and their presentation is that we have used integration by parts as the basic tool rather than Stein's method.
One can obtain Chatterjee's result on the nose by carrying out integration by parts \emph{once}.
However, in our derivation of \cref{prop:theoremA} we have used it \emph{twice}.
This has the added benefit that now the Wilson loop observables decouple from the plaquette variables and one no longer has mergers (or ``deformations'' in the terminology of \parencite{Chatterjee2019}) between Wilson loops and plaquettes.
\begin{remark}
	\label{rem:theoremA:forGeneralizedWilsonLoops}
	In \cref{prop:theoremA} and its corollaries, we have assumed that \( W_{l_1} \dotsm W_{l_q} \) are single-argument Wilson loops with respect to the same representation.
	A careful inspection of the calculation reveals that, with minor modifications, those results generalize to \emph{generalized} Wilson loops in possibly different representations.
	For example, the merger of two Wilson loops with different representations \( \rho \) and \( \rho' \) is defined by essentially the same formula as in \cref{def:MergingLoopsGenerally} with the only difference that one trace is taken with respect to \( \rho \) and the other one with respect to \( \rho' \).
	Similarly, the eigenvalue \( \lambda \) in \cref{prop:theoremA} may now depend on the Wilson loop so that the factor \( \lambda \sum_{r=1}^q n_r \) needs to be replaced by \( \sum_{r=1}^q n_r \lambda_r \).
\end{remark}

By \cref{prop:wilsonLoops:expandInLinearLoops}, we have seen that every polynomial Wilson loop can be written as a linear combination of linear generalized loops with respect to different representations.
Note that for linear loops \cref{prop:theoremA} simplifies as there is no longer a twisting term.
In particular, for the Haar measure and the Brownian motion, the relations in \cref{prop:theoremA:haar,prop:theoremA:brownian} simplify to recursion relations involving less and less products of loops.
This observation can be used to calculate the expectation values of the product of arbitrary Wilson loops \( W_{l_1}, \dotsc, W_{l_q} \) according to the following algorithm.
\begin{enumerate}
	\item
		Expand each Wilson loop \( W_{l_i} \) in terms of linear generalized loops (with respect to irreducible representations) as in \cref{prop:wilsonLoops:expandInLinearLoops}.
	\item
		For linear loops, solve the recursion relation in \cref{prop:theoremA:haar,prop:theoremA:brownian} by induction over the number of loops involved.
	\item
		For a single linear loop \( W_{\rho, (c)} \) with respect to a non-trivial irreducible representation \( (\rho, V) \), the expectation value can be calculated using the results of the next section.
		In particular, the expectation
		\begin{equation}
			\int_G W_{\rho, (c)} \dif g = \tr_V\Bigl(c \int_G \rho(g) \dif g\Bigr)
		\end{equation}
		vanishes since \( \int_G \rho(g) \dif g \) is the projection onto \( V^G = \set{0} \).
		This serves as the induction start in the case of the Haar measure.

		For the Brownian motion, \cref{prop:matrixCoefficients:brownian} below implies that
		\begin{equation}
			\int_G W_{\rho, (c)} \, p_t \dif g = \tr_V\Bigl(c \int_G \rho(g) \, p_t \dif g\Bigr) = \exp\Bigl(\frac{1}{2} c_\rho t\Bigr) \tr_V(c),
		\end{equation}
		where \( c_\rho \) is the Casimir eigenvalue.
\end{enumerate}
The following example illustrates this algorithm for the simplest case, namely \( G = \UGroup(1) \).
\begin{example}[Circle group]
	The irreducible representations of \( \UGroup(1) \) are one dimensional and given by \( \rho_n(z) = z^n \) for some \( n \in \Z \).
	Thus, irreducible Wilson loops are of the form \( W_{n, c} = c z^n \) with \( n \in \Z \) and \( c \in \C \).
	Note that the Casimir invariant of \( \rho_n \) is \( - n^2 \).
	
	As an example, let us calculate the expectation value of the product of two arbitrary Wilson loops \( W_1 \) and \( W_2 \) according to the above algorithm.
	\begin{enumerate}
		\item
			The expansion of \( W_1 \) according to \cref{prop:wilsonLoops:expandInLinearLoops}, in this case, just amounts to writing it as a Fourier series:
			\begin{equation}
				W_1 = \sum_{n=-\infty}^{\infty} W_{n, c_{1, n}},
			\end{equation}
			where only finitely many constants \( c_{1,n} \in \C \) are non-zero.
			Similarly, for \( W_2 \) with constants \( c_{2,n} \).
		\item
			By \cref{prop:theoremA:haar}, we have
			\begin{equation}
				- (n^2 + m^2) \int_{\UGroup(1)} W_{n, c}(z) W_{m, d}(z) \dif z = - 2 \int_{\UGroup(1)} \mathcal{M}(W_{n, c}, W_{m, d})(z) \dif z.
			\end{equation}
			The merger is given by \( \mathcal{M}(W_{n, c}, W_{m, d})(z) = c (\I n) z^n \cdot d (\I m) z^m \), \cf \cref{def:MergingLoopsGenerally,rem:theoremA:forGeneralizedWilsonLoops}.
			In line with \cref{prop:merging:againWilsonLoop}, this is again a generalized Wilson loop, namely \( \mathcal{M}(W_{n, c}, W_{m, d}) = W_{n+m, - cd nm} \).
		\item
			Clearly, \( \int_{\UGroup(1)} W_{n,c} \dif z \) vanishes except if \( n = 0 \).
	\end{enumerate}
	Thus, in summary,
	\begin{equation}\begin{split}
		\int_{\UGroup(1)} W_1(z) W_2(z) \dif z 
			&= \sum_{n,m=-\infty}^\infty \int_{\UGroup(1)} W_{n, c_{1, n}} W_{m, c_{2, m}}(z) \dif z  \\
			&= \sum_{n,m=-\infty}^\infty \frac{2}{n^2 + m^2} \int_{\UGroup(1)} W_{n+m, - c_{1,n}c_{2,m} nm}(z) \dif z \\
			&= \sum_{n=-\infty}^\infty c_{1,n}c_{2,-n},
	\end{split}\end{equation}
	which, of course, coincides with the result one gets by a direct calculation.
\end{example}

The same strategy can be used to calculate the mixed moments of the random variable \( g \mapsto \tr_\rho(g^k) \).
In this case, the expansion of \( \tr_\rho(g^k) \) as a linear combination of linear Wilson loops can be achieved by decomposing the \( k \)-th tensor power into irreducible components.
For the latter, Schur--Weyl duality can be used and then the second step in the above algorithm essentially boils down to an orthogonality relation of the characters of the dual group.
For Haar distributed variables, this recovers \parencites[Theorem~2]{DiaconisShahshahani1994}[Theorem~2.1]{DiaconisEvans2001} for the unitary group, and \parencite[Theorem~3]{HughesRudnick2003} for the orthogonal and symplectic group.
We leave the details to the reader.

\begin{remark}
	The same algorithm does not work when the Wilson probability measure is added, because the right-hand side of~\eqref{eq:theoremA:wilson} contains terms with more Wilson loops than the original integral.
	In fact, it is a notoriously hard problem to calculate Wilson loop expectation values with respect to the Wilson action, and one has to resort to certain limits to obtain a reasonable result.
	The effect of the two most common limits, namely the strong-coupling expansion and the large \( N \)-limit for \( \UGroup(N) \) or \( \SUGroup(N) \), is readily apparent from~\eqref{eq:theoremA:wilson}.
	In the strong-coupling limit \( \beta \to 0 \), the additional terms with more Wilson loops are suppressed.
	Similarly, the merging and twisting terms as well as \( \lambda \) scale with \( N \) or simplify in the large \( N \)-limit.
	This limit has been extensively studied in \parencite{Chatterjee2019,Jafarov2016}. 
\end{remark}

\section{Polynomials of matrix coefficients}

In this section, we discuss how the basic integration by parts formula of \cref{prop:IntegrationByParts} can be used to determine polynomials of matrix coefficients.

As before, let \( G \) be a compact connected Lie group and \( \nu \) be a probability density with respect to the normalized Haar measure on \( G \).
For a (not necessarily irreducible) real or complex representation \( \rho: G \to \GL(V) \) of \( G \) on a finite-dimensional vector space, define \( T(\nu): V \to V \) by
\begin{equation}
	\label{eq:matrixCoefficients:definition}
	T(\nu) = \int_G \rho(g) \, \nu(g) \dif g \, .
\end{equation}
\begin{example}
	Let \( \varrho \) be a representation of \( G \) on the vector space \( W \).
	Usually, this is taken to be the fundamental representation of \( G \).
	Consider the tensor representation \( \rho = \varrho^{\tensorProd n, \tensorProd n'} = \varrho \tensorProd \dotsm \tensorProd \varrho \tensorProd \varrho^* \tensorProd \dotsm \tensorProd \varrho^* \) on \( V = W^{\tensorProd n} \tensorProd (W^*)^{\tensorProd n'} \) with \( n \) factors of \( \varrho \) and \( n' \) factors of the dual representation \( \varrho^*(g) = \rho(g^{-1})^* \).
	Using bold-face multi-indices \( \mathbf{i} = (i; i') = (i^{}_1, \dotsc, i^{}_n; i'_1, \dotsc, i'_{n'}) \) to denote the components of elements of \( W^{\tensorProd n} \tensorProd (W^*)^{\tensorProd n'} \), we find
	\begin{equation}
		\label{eq:MatrixCoefficients:relationWithTensorProdRep}
		T(\nu)_{\mathbf{ij}} = \int_G g^{}_{i_1 j_1} \dotsm g^{}_{i_n j_n} \; g^{-1}_{j'_1 i'_1} \dotsm g^{-1}_{j'_{n'} i'_{n'}} \, \nu \,  \dif g \,,
	\end{equation}
	where \( g_{kl} = \varrho(g)_{kl} \) are the matrix coefficients of \( g \) in the representation \( \varrho \).
	Thus, in this case, \( T(\nu) \) completely encodes the \( \nu \)-expectation value of polynomials in matrix coefficients and their inverses.
	The formulation in terms of the tensor product linearizes the problem of determining the polynomial coefficients on \( G \) to a study of the linear operator \( T(\nu) \).
\end{example}
Somewhat surprisingly the simple integration by parts formula of \cref{prop:IntegrationByParts} combined with basic representation theory of compact Lie groups allows us to determine \( T(\nu) \).
Before we discuss this in detail, let us recall the isotypic decomposition.
Consider a representation \( \rho \) of \( G \) on a vector space \( V \).
Since \( G \) is compact, \( V \) decomposes into a direct sum of irreducible representations \( V_\tau \), see, \eg, \parencite[Corollary~IV.4.7]{Knapp2002}\footnote{\Textcite{Knapp2002} only discusses the case of a complex representation. The proof in the real case is almost identical except that one uses a \( G \)-invariant real inner product.}.
For a given irreducible representation \( \tau \), define the isotypic component \( V_{\equivClass{\tau}} \) to be the sum of all \( V_{\tau'} \) for which \( \tau' \) is equivalent to \( \tau \); with the convention that \( V_{\equivClass{\tau}} = \set{0} \) if there is no such subrepresentation.
The resulting direct sum decomposition
\begin{equation}
	\label{eq:matrixCoefficients:isotypicDecomp}
	V = \bigoplus_{\equivClass{\tau} \in \hat{G}} V_{\equivClass{\tau}}
\end{equation}
is called the \emph{isotypic decomposition}.
Here the sum is over the set \( \hat{G} \) of equivalence classes of irreducible representations of \( G \).
Note that the isotypic component corresponding to the trivial representation is the set \( V^G \) of invariant elements. 
\begin{theorem}
	\label{prop:matrixCoefficients:general}
	Let \( G \) be a compact connected Lie group and \( \nu \) be a probability density with respect to the normalized Haar measure on \( G \).
	For a real or complex representation \( \rho: G \to \GL(V) \) of \( G \), the operator \( T(\nu): V \to V \) defined in~\eqref{eq:matrixCoefficients:definition} respects the isotypic decomposition~\eqref{eq:matrixCoefficients:isotypicDecomp} and satisfies
	\begin{equation}
		c_{\equivClass{\tau}} T(\nu)_{|V_{\equivClass{\tau}}} = T(\Delta \nu)_{|V_{\equivClass{\tau}}}
	\end{equation}
	for each irreducible subrepresentation \( \tau \), where \( c_{\equivClass{\tau}} \in \R \) are non-positive constants depending only on the isomorphism type of the representation \( \tau \).
	Moreover, \( c_{\equivClass{\tau}} = 0 \) if and only if \( \tau \) is the trivial representation.
\end{theorem}
\begin{proof}
	\Cref{prop:IntegrationByParts} applied to the matrix coefficients \( g_{ij} = \rho(g)_{ij} \) yields
	\begin{equation}
		\int_G g_{ij} \, \Delta \nu(g) \, \dif g
			= \int_G \Delta g_{ij} \, \nu(g) \, \dif g
			= K_{lkkj} \int_G g_{il} \, \nu(g) \, \dif g,
	\end{equation}
	where the second equality follows from~\eqref{eq:EigenvalueK-OperatorGeneral:calculation}.
	Rewriting this equality in terms of operators gives \( T(\Delta \nu) = T(\nu) \rho(C) \) with \( \rho(C) \) being the Casimir invariant.
	
	By going back to their definition, \( T(\nu) \) and \( \rho(C) \) respect the decomposition of \( V \) into irreducible representations and so also the isotypic decomposition.
	We have to show that the Casimir invariant \( \rho(C) \) acts as a scalar multiple of the identity on each irreducible component \( V_\tau \) and that the corresponding eigenvalue \( c_\tau \) is real and non-positive.
	If \( \tau \) is a complex representation, this is exactly \parencite[Proposition~IX.7.6.4]{Bourbaki2005}.
	Moreover, \( c_\tau = 0 \) if and only if \( \tau \) is the trivial representation.
	By the same proposition, \( c_\tau \) can be expressed in terms of the highest weight associated with \( \tau \) and so it only depends on the isomorphism type of the representation \( \tau \).
	
	Now, for an irreducible real representation \( \tau: G \to \End(V_\tau) \), we can pass to its complexification \( \tau^\C: G \to \End(V^\C_\tau) \).
	Clearly, the complex-linear extension of the Casimir \( \tau(C) \) is the Casimir \( \tau^\C(C) \) of the complexified representation.
	By \parencite[Theorem~6.3 and Proposition~6.6]{BrockerDieck1985}, the representation \( V^\C \) is either irreducible or a direct sum of the form \( U \oplus \bar{U} \) or \( U \oplus U \) for an irreducible complex representation \( U \).
	In either case, the above argument shows that the Casimir \( \tau^\C(C) \) acts as a scalar multiplication, because the Casimir eigenvalue of the complex conjugate representation \( \bar{U} \) is the same as the one of the representation \( U \).
	By restricting to the real part, we conclude that \( \tau(C) \) is a scalar multiple of the identity.
	This finishes the proof.
\end{proof}
\begin{corollary}[Haar measure]
	\label{prop:matrixCoefficients:haar}
	Let \( G \) be a compact connected Lie group.
	For a real or complex representation \( \rho: G \to \GL(V) \) of \( G \), the operator \( T(\nu = 1): V \to V \) defined in~\eqref{eq:matrixCoefficients:definition} is the projection onto \( V^G \) along the isotypic decomposition~\eqref{eq:matrixCoefficients:isotypicDecomp}.
\end{corollary}
\begin{proof}
	For \( \nu = 1 \), we have \( \Delta \nu = 0 \) and so \( c_{\equivClass{\tau}} T(1)_{|V_{\equivClass{\tau}}} = 0 \) for every irreducible subrepresentation \( \tau \).
	Because \( c_{\equivClass{\tau}} \) is strictly negative for non-trivial representations \( \tau \), the restriction \( T(1)_{|V_{\equivClass{\tau}}} \) has to vanish for such representations.
	Finally, the restriction of \( T(1) \) to \( V^G \) is clearly the identity operator.
\end{proof}

\begin{corollary}[Brownian motion]
	\label{prop:matrixCoefficients:brownian}
	Let \( G \) be a compact connected Lie group and let $\rho: G \to V$ be a real or complex representation of \( G \).
	The expectation value of the \( \GL(V) \)-valued random variable \( \rho \) relative to the Riemannian Brownian motion \( (g_t)_{t \geq 0} \) respects the isotypic decomposition~\eqref{eq:matrixCoefficients:isotypicDecomp} and satisfies
	\begin{equation}
		\label{eq:matrixCoefficients:brownianMotion}
		\Expect\bigl(\rho(g_t)\bigr)_{|V_{\equivClass{\tau}}} = \exp\Bigl(\frac{1}{2} c_{\equivClass{\tau}} t\Bigr) \id_{|V_{\equivClass{\tau}}}.
	\end{equation}
	for each irreducible subrepresentation \( \tau \), where \( c_{\equivClass{\tau}} \in \R \) are the same non-positive constants as in \cref{prop:matrixCoefficients:general}.
	Equivalently, \( \Expect\bigl(\rho(g_t)\bigr) = \exp\Bigl(\frac{t}{2} \rho(C) \Bigr) \).
	Moreover,
	\begin{equation}
		\lim_{t \to \infty}\Expect\bigl(\rho(g_t)\bigr) = T(1).
		\qedhere
	\end{equation}
\end{corollary}
\begin{proof}
	By definition, \( \Expect\bigl(\rho(g_t)\bigr) = T(p_t) \) with \( p_t \) being the heat density.
	Using the heat equation~\eqref{eq:brownianMotion:heatEquation}, \cref{prop:matrixCoefficients:general} implies that the expectation value satisfies
	\begin{equation}
		c_{\equivClass{\tau}} \Expect\bigl(\rho(g_t)\bigr)_{|V_{\equivClass{\tau}}} = 2 \frac{\dif}{\dif t} \Expect\bigl(\rho(g_t)\bigr)_{|V_{\equivClass{\tau}}}.
	\end{equation}
	For the initial condition, note that \( \lim_{t \to 0} \Expect\bigl(\rho(g_t)\bigr) = \rho(e) = \id_V \).
	This shows that the expectation value is given by~\eqref{eq:matrixCoefficients:brownianMotion}.
	Since \( c_{\equivClass{\tau}} \) are the eigenvalues of the Casimir operator, we get \( \Expect\bigl(\rho(g_t)\bigr) = \exp\Bigl(\frac{t}{2} \rho(C) \Bigr) \).

	For a non-trivial subrepresentation \( \tau \), the constant \( c_{\equivClass{\tau}} \) is strictly negative and thus
	\( \Expect\bigl(\rho(g_t)\bigr)_{|V_{\equivClass{\tau}}} \) converges to \( 0 \) as \( t \to \infty \).
	On the other hand, \( \Expect\bigl(\rho(g_t)\bigr)_{|V^G} = \id_{V^G} \).
	Thus, in summary, \( \Expect\bigl(\rho(g_t)\bigr) \) converges to the projection onto \( V^G \).
\end{proof} 
In the case of the classical groups \( G = \UGroup(N), \OGroup(N), \SpGroup(N) \), the expectation value formula \( \Expect\bigl(\rho(g_t)\bigr) = \exp\Bigl(\frac{t}{2} \rho(C) \Bigr) \) has been obtained in \parencites[Proposition~2.4]{Levy2008}[Lemma~4.1 and~4.2]{Dahlqvist2017} using the explicit expression of the corresponding Casimir operator.
Moreover, the long time asymptotic behavior has been established in this case using a rather complicated calculation, \cf \parencite[Theorem~4.3 and Lemma~5.1]{Dahlqvist2017}.
In contrast, our proof shows that this is a direct and straight-forward consequence of the non-positivity of the spectrum of the Casimir.  

For the tensor representation, the following result shows that the isotypic decomposition and the constants \( c_{\equivClass{\tau}} \) can be obtained from an eigenvalue problem for an operator determined by the operator \( K \) defined in \cref{eq:K-Operator}.
In particular, the isotypic decomposition of the tensor representation on \( V^{\tensorProd n} \tensorProd (V^*)^{\tensorProd n'} \) for arbitrary integers \( n \) and \( n' \) is completely given in terms of data associated with the tensor representation on \( V \tensorProd V^* \).
This is particularly important for determining the decomposition in concrete examples using computer algebra systems.
\begin{proposition}
	\label{prop:matrixCoefficients:tensorRep}
	Let $G$ be a compact connected Lie group, and let $\varrho: G \to V$ be an irreducible representation of \( G \).
	Let $\lambda \in \R$ be the eigenvalue of the Casimir invariant \( \varrho(C) \).
	For non-negative integers \( n \) and \( n' \), the isotypic decomposition of the tensor representation \( \varrho^{\tensorProd n, \tensorProd n'} = \varrho(g)^{\tensorProd n} \tensorProd \bigl(\varrho(g^{-1})^*\bigr)^{\tensorProd n'} \) on \( V^{\tensorProd n} \tensorProd (V^*)^{\tensorProd n'} \) coincides with the eigenspace decomposition of the operator
	\begin{equation}\begin{split}
		C_{\mathbf{i}\mathbf{j}} &= (n + n') \lambda \, \delta_{\mathbf{ij}}
			\\
			&\quad - 2 \sum_{\substack{r, s = 1 \\ r < s}}^n K_{i_r j_r i_s j_s} \delta_{i_1 j_1} \dotsm \hat{r} \dotsm \hat{s} \dotsm \delta_{i_n j_n} \delta_{i' j'}
			\\
			&\quad - 2 \sum_{\substack{r, s = 1 \\ r < s}}^{n'} K_{j'_r i'_r j'_s i'_s} \delta_{i'_1 j'_1} \dotsm \hat{r} \dotsm \hat{s} \dotsm \delta_{i'_{n'} j'_{n'}} \delta_{i j}
			\\
			&\quad + 2 \sum_{r = 1}^n \sum_{s = 1}^{n'} K_{i^{}_r j^{}_r j'_s i'_s} \delta_{i^{}_1 j^{}_1} \dotsm \hat{r} \dotsm \delta_{i^{}_n j^{}_n} \delta_{i'_1 j'_1} \dotsm \hat{s} \dotsm \delta_{i'_{n'} j'_{n'}}.
	\end{split}\end{equation}
	Moreover, the constants \( c_{\equivClass{\tau}} \) of \cref{prop:matrixCoefficients:general} are equal to the corresponding eigenvalues of \( C \).
\end{proposition}
\begin{proof}
	As discussed above, the isotypic decomposition coincides with the eigenspace decomposition of the Casimir element \( \varrho^{\tensorProd n, \tensorProd n'}(C) \).
	To calculate the components \( C_{\mathbf{i j}} \) of this operator, note that 
	\begin{equation}
		\xi_{\mathbf{i} \mathbf{j}}	= 
		\sum_{r = 1}^n \xi_{i^{}_r j^{}_r} \delta_{i^{}_1 j^{}_1} \dotsm \hat{r} \dotsm \delta_{i^{}_n j^{}_n} \delta_{i' j'}
		- \sum_{r = 1}^{n'} \xi_{j'_r i'_r} \delta_{i j} \delta_{i'_1 j'_1} \dotsm \hat{r} \dotsm  \delta_{i'_{n'} j'_{n'}}
	\end{equation}
	for \( \xi \in \LieA{g} \).
	Consequently, the operator $K$ (see \cref{eq:K-Operator}) for the tensor representation \( \varrho^{\tensorProd n, \tensorProd n'} \) takes the form
	\begin{equation}\begin{split}
		K_{\mathbf{i j k l}} &=
			\sum_{r, s = 1}^n K_{i_r j_r k_s l_s} \delta_{i^{}_1 j^{}_1} \dotsm \hat{r} \dotsm \delta_{i^{}_n j^{}_n} \delta_{k^{}_1 l^{}_1} \dotsm \hat{s} \dotsm \delta_{k^{}_n l^{}_n} \delta_{i' j'} \delta_{k' l'} 
			\\
			&\quad + \sum_{r, s = 1}^{n'} K_{j'_r i'_r l'_s k'_s} \delta_{i'_1 j'_1} \dotsm \hat{r} \dotsm \delta_{i'_{n'} j'_{n'}} \delta_{k'_1 l'_1} \dotsm \hat{s} \dotsm \delta_{k'_{n'} l'_{n'}} \delta_{i j} \delta_{k l}
			\\
			&\quad - \sum_{r = 1}^n \sum_{s = 1}^{n'} K_{i^{}_r j^{}_r l'_s k'_s} \delta_{i^{}_1 j^{}_1} \dotsm \hat{r} \dotsm \delta_{i^{}_n j^{}_n} \delta_{k'_1 l'_1} \dotsm \hat{s} \dotsm \delta_{k'_{n'} l'_{n'}} \delta_{i' j'} \delta_{k l} 
			\\
			&\quad - \sum_{r = 1}^{n'} \sum_{s = 1}^{n} K_{j'_r i'_r k^{}_s l^{}_s} \delta_{i'_1 j'_1} \dotsm \hat{r} \dotsm \delta_{i'_{n'} j'_{n'}} \delta_{k^{}_1 l^{}_1} \dotsm \hat{s} \dotsm \delta_{k^{}_n l^{}_n} \delta_{i j} \delta_{k' l'} 
	\end{split}\end{equation}
	Our objective is to calculate the Casimir element \( C_{\mathbf{i j}} = K_{\mathbf{i k k j}} \), with implicit summation over \( \mathbf{k} \) understood.
	For this purpose, notice that, for each \( r \neq s \) and with summation over \( k \), we have
	\begin{equation}
		K_{i_r k_r k_s j_s} \delta_{i_1 k_1} \dotsm \hat{r} \dotsm \delta_{i_n k_n} \delta_{k_1 j_1} \dotsm \hat{s} \dotsm \delta_{k_n j_n}
			= K_{i_r j_r i_s j_s} \delta_{i_1 j_1} \dotsm \hat{r}, \hat{s} \dotsm \delta_{i_n j_n}.
	\end{equation}
	On the other hand, for \( r = s \), we obtain
	\begin{equation}
		K_{i_r k_r k_r j_r} \delta_{i_1 k_1} \dotsm \hat{r} \dotsm \delta_{i_n k_n} \delta_{k_1 j_1} \dotsm \hat{r} \dotsm \delta_{k_n j_n}
			= \lambda \delta_{ij}.
	\end{equation}
	Using these and similar identities in each of the four summands yields
	\begin{equation}\begin{split}
		C_{\mathbf{i j}} &= K_{\mathbf{i k k j}}
			\\	
			&=
			\sum_{\substack{r, s = 1 \\ r \neq s}}^n K_{i_r j_r i_s j_s} \delta_{i_1 j_1} \dotsm \hat{r}, \hat{s} \dotsm \delta_{i_n j_n} \delta_{i' j'} + n \lambda \delta_{ij} \delta_{i' j'}
			\\
			&\quad + \sum_{\substack{r, s = 1 \\ r \neq s}}^{n'} K_{j'_r i'_r j'_s i'_s} \delta_{i'_1 j'_1} \dotsm \hat{r}, \hat{s} \dotsm \delta_{i'_{n'} j'_{n'}} \delta_{i j} + n' \lambda \delta_{ij} \delta_{i' j'}
			\\
			&\quad - \sum_{r = 1}^n \sum_{s = 1}^{n'} \bigl( K_{i^{}_r j^{}_r j'_s i'_s} + K_{j'_s i'_s i^{}_r j^{}_r}  \bigr)\delta_{i^{}_1 j^{}_1} \dotsm \hat{r} \dotsm \delta_{i^{}_n j^{}_n} \delta_{i'_1 j'_1} \dotsm \hat{s} \dotsm \delta_{i'_{n'} j'_{n'}}
			\\
			&=
			(n + n') \lambda \, \delta_{\mathbf{ij}}
			\\
			&\quad + 2 \sum_{\substack{r, s = 1 \\ r < s}}^n K_{i_r j_r i_s j_s} \delta_{i_1 j_1} \dotsm \hat{r} \dotsm \hat{s} \dotsm \delta_{i_n j_n} \delta_{i' j'}
			\\
			&\quad + 2 \sum_{\substack{r, s = 1 \\ r < s}}^{n'} K_{j'_r i'_r j'_s i'_s} \delta_{i'_1 j'_1} \dotsm \hat{r} \dotsm \hat{s} \dotsm \delta_{i'_{n'} j'_{n'}} \delta_{i j}
			\\
			&\quad - 2 \sum_{r = 1}^n \sum_{s = 1}^{n'} K_{i^{}_r j^{}_r j'_s i'_s} \delta_{i^{}_1 j^{}_1} \dotsm \hat{r} \dotsm \delta_{i^{}_n j^{}_n} \delta_{i'_1 j'_1} \dotsm \hat{s} \dotsm \delta_{i'_{n'} j'_{n'}}.
	\end{split}\end{equation}
	This finishes the proof.
\end{proof}

In applications, one can often use invariant theory to obtain a spanning set for the space of invariants \( \bigl(V^{\tensorProd n} \tensorProd (V^*)^{\tensorProd n'}\bigr)^G \).
This is well-studied for the classical groups, see \parencite{GoodmanWallach2009}, and also for some exceptional groups, see for example \parencite{Schwarz1988} for the group \( G = G_2\) and its 7-dimensional irreducible representation.
The following theorem shows that such a spanning set is already enough to calculate the operator \( T \) for the Haar measure (\( \nu = 1 \)).
This generalizes the main results of \parencite{Collins2002,CollinsSniady2006} for the classical groups \( G = \UGroup(N), \OGroup(N), \SpGroup(N) \) to arbitrary compact Lie groups.
\begin{theorem}
	\label{prop:matrixCoefficients:tensorRepSchur}
	Let $G$ be a compact Lie group and let $\rho: G \to V$ be a finite-dimensional real or complex representation of \( G \) leaving the inner product \( \scalarProd{\cdot}{\cdot} \) on \( V \) invariant.
	Let \( \mathcal{A} \) be a finite-dimensional inner product space over the same field as \( V \) and let \( \tau: \mathcal{A} \to V^{\tensorProd n} \tensorProd (V^*)^{\tensorProd n'} \) be a linear map.
	Denote by \( \tau^*: V^{\tensorProd n} \tensorProd (V^*)^{\tensorProd n'} \to \mathcal{A} \) the adjoint of \( \tau \) with respect to the following inner product\footnotemark{} on \( V^{\tensorProd n} \tensorProd (V^*)^{\tensorProd n'} \):
	\footnotetext{Under the identification \( V^{\tensorProd n} \tensorProd (V^*)^{\tensorProd n'} \isomorph \Hom\bigl(V^{\tensorProd n'}, V^{\tensorProd n}\bigr) \), this inner product corresponds to the inner product
	\begin{equation}
		\scalarProd{S_1}{S_2} = \tr\bigl(S_2^* S_1\bigr), \qquad S_1, S_2 \in \Hom\bigl(V^{\tensorProd n'}, V^{\tensorProd n}\bigr).
	\end{equation}
	}
	\begin{equation}
		\scalarProd{v \tensorProd \alpha}{w \tensorProd \beta} = \scalarProd{v_1}{w_1} \dotsm \scalarProd{v_n}{w_n} \, \scalarProd{\alpha_1}{\beta_1} \dotsm \scalarProd{\alpha_{n'}}{\beta_{n'}}.
	\end{equation}
	There exists a unique map \( \mathrm{Wg}: \mathcal{A} \to \mathcal{A} \) satisfying the following properties:
	\begin{enumerate}
		\item \( \tau^* \circ \tau \circ \mathrm{Wg} \circ \tau^* \circ \tau = \tau^* \circ \tau \),
		\item \( \mathrm{Wg} \circ \tau^* \circ \tau \circ \mathrm{Wg} = \mathrm{Wg} \),
		\item \( \mathrm{Wg}^* \circ \tau^* \circ \tau = \tau^* \circ \tau \circ \mathrm{Wg} \),
		\item \( \tau^* \circ \tau \circ \mathrm{Wg}^* = \mathrm{Wg} \circ \tau^* \circ \tau \).
	\end{enumerate}
	If the image of \( \tau \) is \( \bigl(V^{\tensorProd n} \tensorProd (V^*)^{\tensorProd n'}\bigr)^G \), then
	\begin{equation}
		\label{eq:matrixCoefficients:tensorRepSchur:T}
		T (1) = \tau \circ \mathrm{Wg} \circ \tau^*,
	\end{equation}
	with \( T \) defined as in~\eqref{eq:matrixCoefficients:definition} relative to the tensor representation \( \rho^{\tensorProd n, \tensorProd n'} \).
	In particular, the coefficients of \( T(1) \) with respect to an orthonormal basis of \( V \) are given by
	\begin{equation}
		\label{eq:matrixCoefficients:tensorRepSchur:TCoeff}
		T (1)_{\mathbf{ij}} = \sum_{k, l} \tau(a_k)_\mathbf{i} \, \overline{\tau(a_l)_\mathbf{j}} \, \scalarProd{\, \mathrm{Wg} (a_k)}{a_l},
	\end{equation}
	where \( \set{a_k} \) is an orthonormal basis of \( \mathcal{A} \).
\end{theorem}
For the fundamental representation \( \rho \) of \( G = \UGroup(N) \), as we will discuss in detail in \cref{sec:examples:ugroup}, a generating set of \( G \)-invariant elements of \( V^{\tensorProd n} \tensorProd (V^*)^{\tensorProd n} \) is given in terms of permutations.
That is, an orthonormal basis of \( \mathcal{A} \) is indexed by permutations \( \sigma \in S_n \) and the expression \( \scalarProd{\, \mathrm{Wg} (\sigma)}{\varsigma} \) in~\eqref{eq:matrixCoefficients:tensorRepSchur:TCoeff} recovers the so-called Weingarten function on \( S_n \).
For this reason, we will refer to \( \mathrm{Wg} \) as the Weingarten map for the group \( G \) (relative to \( \tau \)).
\begin{proof}
	Recall that the pseudoinverse (or Moore--Penrose inverse) of an operator \( A: H_1 \to H_2 \) between finite-dimensional inner product spaces is an operator \( A^+: H_2 \to H_1 \) satisfying
	\begin{enumerate}
		\item \( A A^+ A = A \),
		\item \( A^+ A A^+ = A^+ \),
		\item \( A A^+ \) and \( A^+ A \) are self-adjoint.
	\end{enumerate}
	It is well known that every operator has a unique pseudoinverse (in the finite-dimensional setting).
	Moreover, the pseudoinverse satisfies \( A^+ = (A^* A)^+ A^* \) and the operator \( A A^+: H_2 \to H_2 \) is the orthogonal projector onto the image of \( A \).

	Now the properties (1) to (4) entail that \( \mathrm{Wg} \) is the pseudoinverse of \( \tau^* \circ \tau \).
	In particular, such an operator \( \mathrm{Wg} \) exists and is uniquely defined by these properties.
	Moreover, \( \tau^+ = (\tau^* \circ \tau)^+ \circ \tau^* = \mathrm{Wg} \circ \tau^* \).
	Hence,
	\begin{equation}
		\tau \circ \tau^+ = \tau \circ \mathrm{Wg} \circ \tau^*
	\end{equation}
	is the orthogonal projector onto the image of \( \tau \), which is \( \bigl(V^{\tensorProd n} \tensorProd (V^*)^{\tensorProd n'}\bigr)^G \) by assumption.

	On the other hand, \cref{prop:matrixCoefficients:haar} shows that \( T(1) \) is the projector onto \( \bigl(V^{\tensorProd n} \tensorProd (V^*)^{\tensorProd n'}\bigr)^G \) along the isotypic decomposition.
	Since \( \rho \) leaves the inner product \( \scalarProd{\cdot}{\cdot} \) invariant, \( T(1) \) is easily seen to be self-adjoint.
	Thus, \( T(1) \) is an orthogonal projector onto \( \bigl(V^{\tensorProd n} \tensorProd (V^*)^{\tensorProd n'}\bigr)^G \) and thus coincides with \( \tau \circ \mathrm{Wg} \circ \tau^* \).
\end{proof}
\begin{remark}
	\label{rem:matrixCoefficients:weingartenByDiagonalizing}
	The proof shows that the Weingarten map \( \mathrm{Wg} \) is the pseudoinverse of \( \tau^* \circ \tau \).
	This observation can be used to calculate \( \mathrm{Wg} \) using one of the well-known constructions of a pseudoinverse.
	For example, one could exploit the fact that \( \tau^* \circ \tau \) is self-adjoint as follows.
	By the spectral theorem, we can write \( \tau^* \circ \tau = U D U^* \) for a unitary operator \( U \) and a diagonal matrix \( D \).
	Reordering the entries of \( D \), we may assume that \( D = \mathrm{diag}(\lambda_1, \dotsc, \lambda_k, 0, \dotsc, 0) \) where \( \lambda_k \in \R \) is non-zero.
	Then
	\begin{equation}
		\mathrm{Wg} = U \, \mathrm{diag}(\lambda_1^{-1}, \dotsc, \lambda_k^{-1}, 0, \dotsc, 0) \, U^*
	\end{equation}
	is the pseudoinverse of \( \tau^* \circ \tau \).
	For the fundamental representation of the classical groups \( G = \UGroup(N), \OGroup(N), \SpGroup(N) \), the decomposition \( \tau^* \circ \tau = U D U^* \) can be calculated using character theory of a certain associated finite group (the Schur--Weyl dual group).
	In this way, we recover the description \parencite[Proposition~2.3 and~3.10]{CollinsSniady2006} of the Weingarten map in these cases.
\end{remark}
Of course, expectation values of products of Wilson loops can be calculated, at least in principle, once all polynomials in matrix coefficients are known.
Thus, one could use \cref{prop:matrixCoefficients:general} for the tensor representation to establish the factorization \cref{prop:theoremA}.
On the other hand, \cref{prop:theoremA} applied to well-chosen Wilson loops yields \cref{prop:matrixCoefficients:general}, showing that these two theorems are hence equivalent.
The following remark explains this in more detail.
\begin{remark}
	For every \( 1 \leq i,j \leq \dim V \), let \( D: V \to V \) be defined by \( D_{pq} = \delta_{pj} \delta_{qi} \) relative to a chosen basis of \( V \).
	That is, \( D \) is the matrix whose only non-zero entry is in the \( j \)-th column in the \( i \)-th row.
	Then the generalized Wilson loop \( W_l(g) = \tr(D \varrho(g^{\pm 1})) \) evaluates to the matrix element \( W_l(g) = g^{\pm 1}_{ij} \).
	This construction shows that the prescriptions \( W_{l_k}(g) = g^{}_{i_k j_k} \) and \(  W_{l'_k}(g) = g^{-1}_{j'_k i'_k} \) define generalized Wilson loops.
	Using~\eqref{eq:MergingRulesGenerally}, the merging of two such loops is given by
	\begin{subequations}\begin{align}
		\mathcal{M}(l_r,l_s) &= + K_{p j_r q j_s} \, g^{}_{i_r p} \, g^{}_{i_s q}\, , \\
		\mathcal{M}(l_r,l'_s) &= - K_{p j_r j'_s q} \, g^{}_{i_r p} \, g^{-1}_{q i'_s}\, , \\
		\mathcal{M}(l'_r,l'_s) &= + K_{j'_r p j'_s q} \, g^{-1}_{p i'_r} \, g^{-1}_{q i'_s}\, .
	\end{align}\end{subequations}
	Thus \cref{prop:theoremA} (\cf, also \cref{rem:theoremA:forGeneralizedWilsonLoops}) implies
	\begin{equation}\begin{split}
		\lambda (n + n') \, & T_{ii' \, jj'}(\nu) - T_{ii' \, jj'}(\Delta \nu)
			=
			\\
			& - 2 \sum_{\substack{r, s = 1 \\ r < s}}^n K_{p j_r q j_s} T_{i i' \, (j_1 \dotso p \dotso q \dotso j_n) j'}(\nu) 
			\\
			&+ 2 \sum_{r = 1}^n \sum_{s = 1}^{n'} K_{p j_r j'_s q} T_{i i' \, (j_1 \dotso p \dotso j_n) (j'_1 \dotso q \dotso j'_n)}(\nu) 
			\\
			& - 2 \sum_{\substack{r, s = 1 \\ r < s}}^{n'} K_{j'_r p j'_s q} T_{i i' \, j (j'_1 \dotso p \dotso q \dotso j'_n)}(\nu),
	\end{split}\end{equation}
	where \( p \) and \( q \) always occur at the \( r \)-th and \( s \)-th position in the multi-indices, respectively.
	Comparing this equation with \cref{prop:matrixCoefficients:tensorRep} establishes \cref{prop:matrixCoefficients:general}.
\end{remark}

\section{Examples}
\label{sec:examples}

\begin{figure}
\centering
  \includegraphics[scale=0.72]{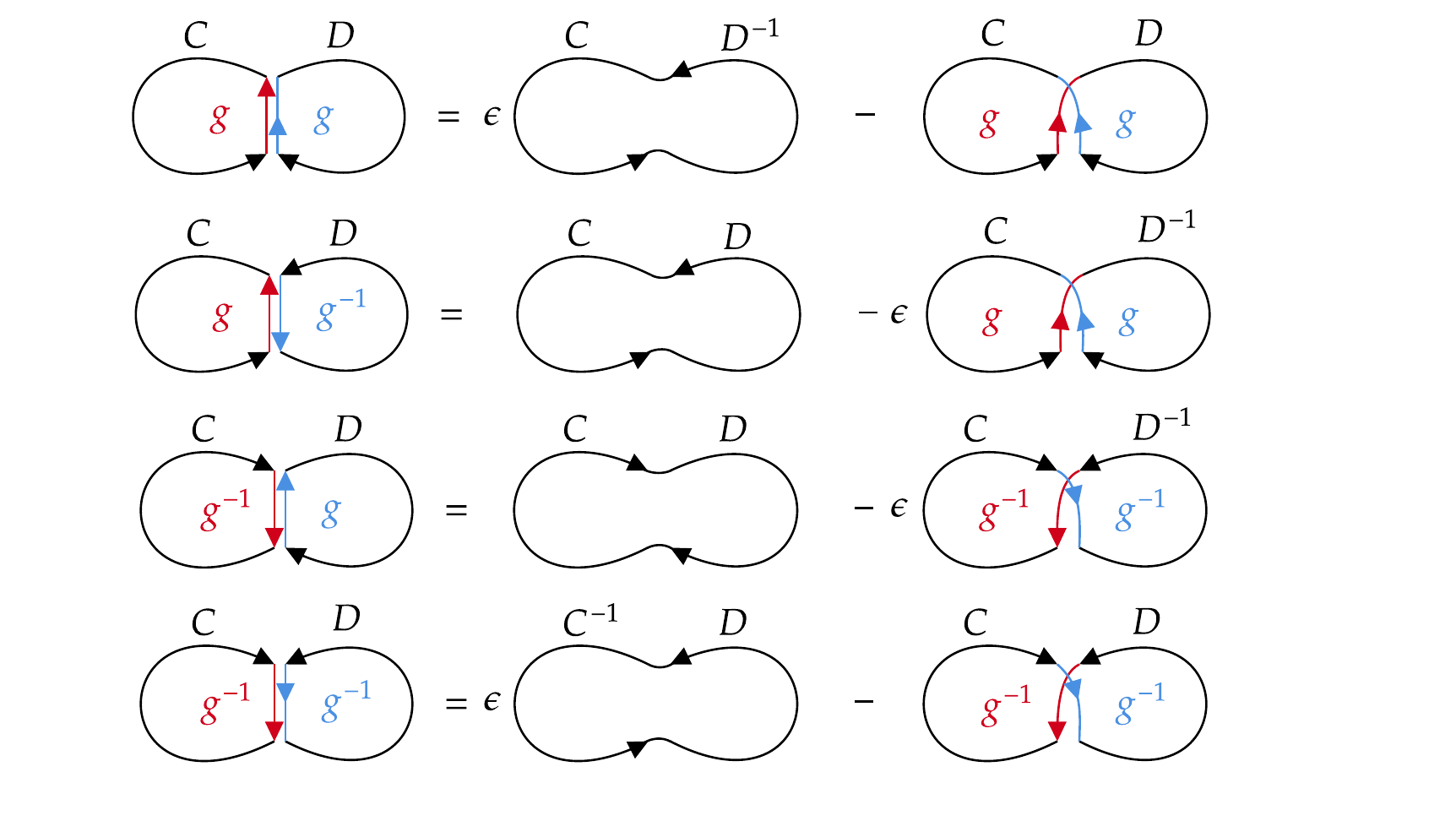}
  \caption{A visualization of the merging rules for the defining representations of $\SpGroup(N)$, $\SOGroup(N)$, and $\UGroup(N)$, see \eqref{eq:MergingRulesSO(N)} and~\eqref{eq:MergingRulesSp(N)}. The sets of rules only differ by the value of the scalar $\epsilon$. One sets $\epsilon = 1$ for $\SpGroup(N)$ and $\SOGroup(N)$, and $\epsilon = 0$ for $\UGroup(N)$. }
\end{figure}

\begin{figure}
\centering
  \includegraphics[scale=0.72]{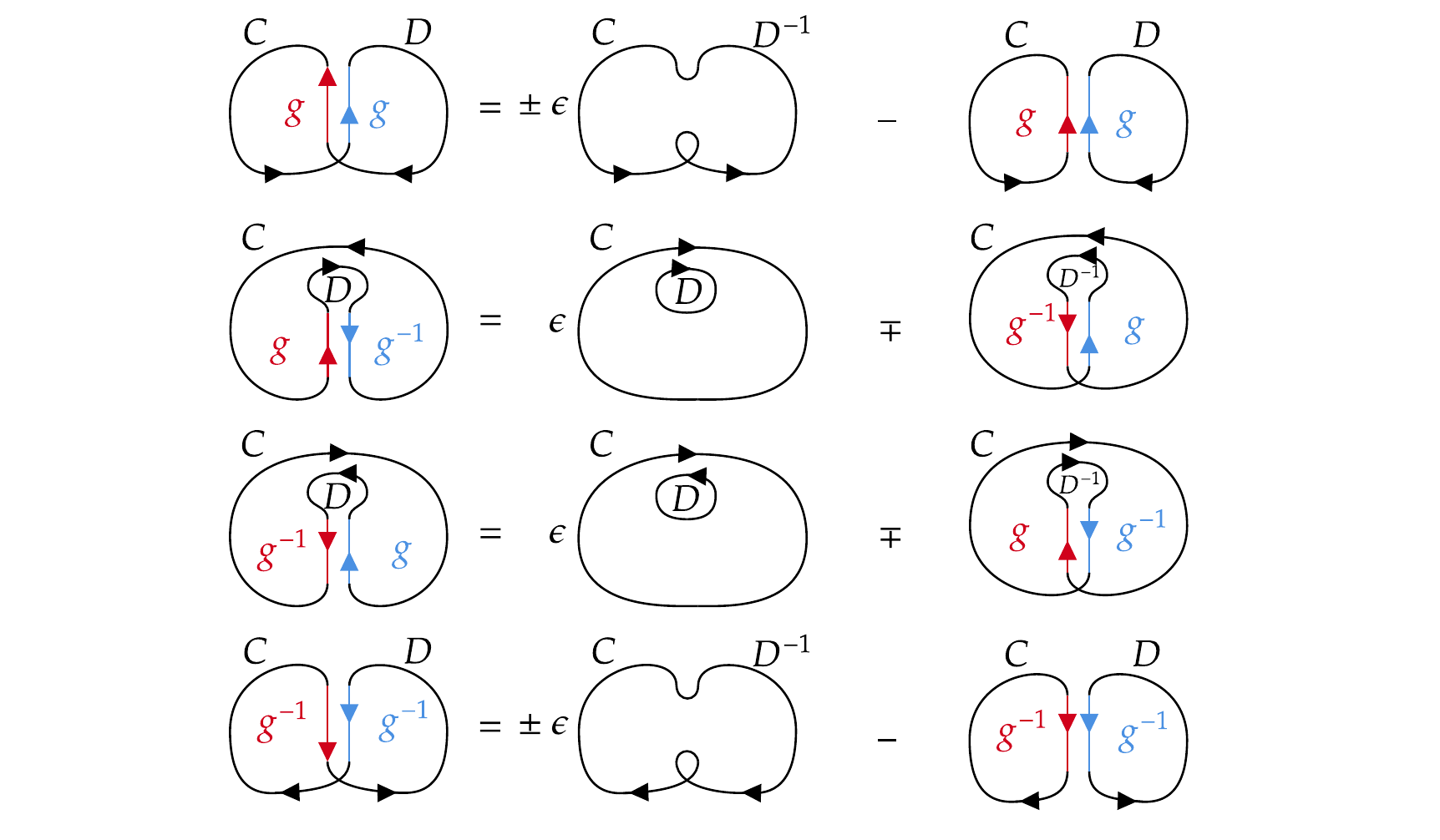}
\caption{A visualization of the twisting rules for the defining representations of $\SOGroup(N), \SpGroup(N)$, and $\UGroup(N)$, see \eqref{eq:TwistingRulesSO(N)} and \eqref{eq:TwistingRulesSp(N)}. The sets of rules only differ by the choice of signs for the $\pm$ and $\mp$, and the value of the scalar $\epsilon$. One chooses the upper signs for $\SOGroup(N)$ and the lower signs for $\SpGroup(N)$ and $\UGroup(N)$; further $\epsilon = 1$ for $\SOGroup(N)$ and $\SpGroup(N)$, and $\epsilon = 0$ for $\UGroup(N)$.}
\end{figure}


\begin{figure}
\centering
  \includegraphics[scale=0.72]{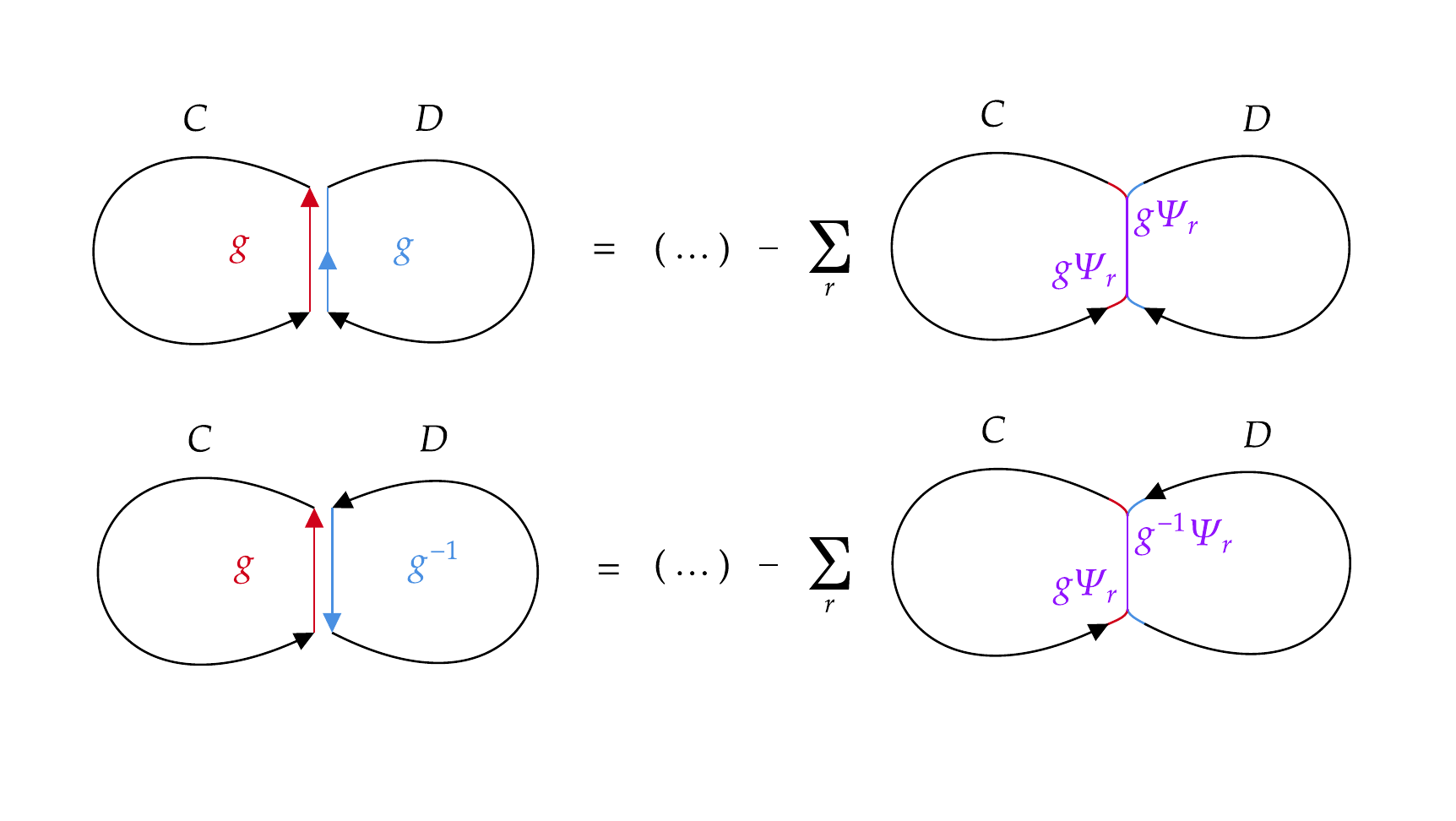}
\caption{A visualization of some merging rules for the 7-dimensional irreducible representation of $G_2$, see \eqref{eq:MergingRulesG_2}. The ellipses represent similar merging rules as the ones for $\SOGroup(N)$, but we also find terms which no longer are expressible as a simple linear combination of Wilson loops.}
\end{figure}

In this section, we explore how \cref{prop:theoremA} reproduces important Wilson loop formulas from \parencite{Chatterjee2019,Jafarov2016} for the groups $G = \SOGroup(N)$ and $\SUGroup(N)$ in a straightforward way, using basic representation-theoretic information rather than the lengthy, explicit calculations employed in the cited papers.
We also study equivalent Wilson loop formulas for other examples that, to the authors' knowledge, do not yet appear in the literature: the classical groups $\SpGroup(N)$ and $\UGroup(N)$, and the exceptional group $G_2$.

\subsection{Defining representation of \texorpdfstring{$\SOGroup(N)$}{SO(N)}}
\label{sec:examples:so}
Let us sketch how \cref{prop:theoremA} reproduces \cite[Thm 8.1]{Chatterjee2019} for $G = \SOGroup(N)$.
Consider the defining representation $\rho : \SOGroup(N) \to \mathbb{R}^{N \times N}$.
A basis $\{\xi^a\}$ of the associated Lie algebra $\mathfrak{so}(N) \subset \mathbb{R}^{N \times N}$ is given in terms of elementary antisymmetric matrices $\tfrac{1}{\sqrt{2}}\left(E_{ij} - E_{ji}\right)$ for $i \neq j$ (where $E_{ij} \in \mathbb{R}^{N \times N}$ with $(E_{ij})_{rs} = \delta_{ir} \delta_{js}$). 
The matrices $\xi^a$ constitute an orthonormal basis of $\mathfrak{so}(N)$ relative to the following inner product  
\begin{equation}
\kappa(X,Y) := -\tr(X \cdot Y) \quad \forall X,Y \in \mathfrak{so}(N),
\end{equation}
which is $-\tfrac{1}{N-2}$ times the Killing form.
A straightforward calculation gives the following completeness relation:
\begin{equation}
\label{eq:SO(N)CompletenessRelation}
\xi^a_{ij} \xi^a_{kl} =  \delta_{ik} \delta_{jl} - \delta_{il} \delta_{jk}.
\end{equation}
Given two Wilson loops $W_l(g) = \tr(C g^{\pm 1}), W_{l'}(g) = \tr(D g^{\pm 1})$ with $C,D \in G$. Then, depending on the given exponents, Equation~\eqref{eq:SO(N)CompletenessRelation} implies:
  \begin{equation}
  \label{eq:MergingRulesSO(N)}
    \mathcal{M}(W_l, W_{l'})(g)
      = \begin{cases}
        (+,+):  \quad \tr(C D^{-1}) - \tr(CgDg), \\
        (+,-):  \quad \tr(C D) - \tr(CgD^{-1}g), \\
        (-,+):  \quad \tr(C D) - \tr( C g^{-1} D^{-1} g^{-1}), \\
        (-,-):  \quad \tr(C^{-1} D) - \tr(C g^{-1} D g^{-1}), \\
      \end{cases}
  \end{equation}
This corresponds to a linear combination of what in \cite{Chatterjee2019} is called the negative and positive mergers $W_{l \ominus_{i,j} l'}, W_{l \oplus_{i,j} l'}$ of the loops $W_l, W_{l'}$.
In an analogous way, all other sums can be handled, and in their notation we find:
	\begin{equation}
		\langle d W_l, d W_{l'} \rangle
		=
		\sum_{\substack{j \in E(l),\\ j' \in E(l')}} \left( W_{l \ominus_{j,j'} l'} - W_{l \oplus_{j,j'} l'} \right).
	\end{equation}
Further, given a single-argument Wilson loop $W_l(g) = \tr(C g^{\pm 1} D g^{\pm 1})$ with $C,D \in G$. Then, depending on the exponents:
  \begin{equation}
  \label{eq:TwistingRulesSO(N)}
    \mathcal{T}(W_l)(g)
      = \begin{cases}
        (+,+):  \quad \tr(C D^{-1}) - \tr(Cg) \tr(Dg), \\
        (+,-):  \quad  -  \tr(g^{-1} C g D^{-1}) + \tr(C) \tr(D) , \\
        (-,+):  \quad  - \tr(C g^{-1} D^{-1} g) + \tr(C) \tr(D), \\
        (-,-):  \quad \tr(C D^{-1}) - \tr(g^{-1} C) \tr(g^{-1} D), \\
      \end{cases}
  \end{equation}	
Again, this corresponds to linear combinations of what in \cite{Chatterjee2019} is called \emph{twistings} $ W_{\propto_{j,j'} l}$ and \emph{splittings} $W_{\times_{j,j'}^1}, W_{\times_{j,j'}^2}$ of $W_l$. 
Recall that the Laplacian can be written as $\Delta = \xi^a \xi^a$. Together with the completeness relation~\eqref{eq:SO(N)CompletenessRelation} this implies:
	\begin{equation}
		\Delta g_{il} = 
		\sum_a \rho(\xi^a)_{ij} \rho(\xi^a)_{jk} g_{kl} = 
		\left(\delta_{ij}\delta_{jk} - \delta_{ik}\delta_{jj}\right)g_{kl} = 
		(1 - N) g_{il}.
	\end{equation} 
Note that by orthogonality, $(g^{-1})_{il} = g_{li}$ is itself just a generic matrix element, so the above equation also holds under the replacement $g \rightsquigarrow g^{-1}$. Thus, the Casimir eigenvalue $\lambda$ equals $(1-N)$.
In the defining representation of $\SOGroup(N)$, \cref{prop:theoremA} implies \cite[Thm 8.1]{Chatterjee2019} as a corollary, the only differences coming from the fact that we used integration by parts \emph{twice} rather than once in our derivation of \cref{prop:theoremA}.
\subsection{Defining representation of \texorpdfstring{$\SpGroup(N)$}{Sp(N)}}
The compact symplectic group \( \SpGroup(N) = \SpGroup(2N, \C) \intersect \UGroup(2N) \) consists of unitary \( 2N \times 2N \) matrices \( M \) satisfying \( M^\T J M = J \), where \( J = \smallMatrix{0 & I_{N} \\ - I_N & 0} \).
Elements of its Lie algebra \( \SpAlgebra(N) \) are block matrices of the form \( \iota(a, b) = \smallMatrix{a & b \\ - \bar{b} & \bar{a}} \) with \( a^* = -a \) and \( b^\T = b \).
Accordingly, a basis of \( \SpAlgebra(N) \) is given by the following elements
\begin{subequations}
  \begin{gather}
    \frac{1}{2} \iota\bigl(E_{ab} - E_{ba}, 0\bigr), \quad
    \frac{1}{2} \iota\bigl(\I E_{ab} + \I E_{ba}, 0\bigr), \quad
    \frac{1}{\sqrt{2}} \iota\bigl(\I E_{cc}, 0\bigr), \\
    \frac{1}{2} \iota\bigl(0, E_{ab} + E_{ba}\bigr), \quad
    \frac{1}{2} \iota\bigl(0, \I E_{ab} + \I E_{ba}\bigr), \quad
    \frac{1}{\sqrt{2}} \iota\bigl(0, E_{cc}\bigr), \quad
    \frac{1}{\sqrt{2}} \iota\bigl(0, \I E_{cc}\bigr)
  \end{gather}
\end{subequations}
where \( 1 \leq a < b \leq N \) and \( 1 \leq c \leq N \) and where the matrix \( E_{ab} \) is defined by \( (E_{ab})_{kl} = \delta_{ak} \delta_{bl} \) as before.
These elements form an orthonormal basis with respect to the inner product
\begin{equation}
  \kappa(X, Y) = -\tr(XY), \qquad X, Y \in \SpAlgebra(N),
\end{equation}
which is \( -\tfrac{1}{2N + 2} \) times the Killing form.
A direct calculation shows that the completeness relation for \( \SpAlgebra(N) \) reads
\begin{equation}
  K_{ijkl} = J_{ik} J_{jl} - \delta_{il} \delta_{jk},
\end{equation}
see also \parencite[Appendix~A]{Dahlqvist2017}.

Thus, the merging of the loops \( W_l(g) = \tr(C g^\sigma) \) and \( W_{l'}(g) = \tr(D g^\varsigma) \) with \( C, D \in \SpGroup(N) \) is given by
\begin{equation}
\label{eq:MergingRulesSp(N)}
  \mathcal{M}(W_l, W_{l'})(g)
    = \begin{cases}
      (+,+):  \quad \tr(C D^{-1}) - \tr(CgDg), \\
      (+,-):  \quad \tr(CD) - \tr(C g D^{-1} g), \\
      (-,+):  \quad \tr(CD) - \tr(C g^{-1} D^{-1} g^{-1}), \\
      (-,-):  \quad \tr(C^{-1} D) - \tr(Cg^{-1}Dg^{-1}), 
    \end{cases}
\end{equation}
depending on the signatures \( (\sigma, \varsigma) \).
Similarly, the twisting of the loop \( W_l(g) = \tr(C g^\sigma D g^\varsigma) \) with \( C, D \in \SpGroup(N) \) takes the form 
\begin{equation}
\label{eq:TwistingRulesSp(N)}
  \mathcal{T}(W_l)(g)
    = \begin{cases}
      (+,+):  \quad -\tr(C D^{-1}) - \tr(Cg) \tr(Dg), \\
      (+,-):  \quad \tr(Cg D^{-1}g^{-1}) + \tr(C) \tr(D), \\
      (-,+):  \quad \tr(Cg^{-1} D^{-1}g) + \tr(C) \tr(D), \\
      (-,-):  \quad -\tr(C D^{-1}) - \tr(Cg^{-1}) \tr(Dg^{-1}). 
    \end{cases}
\end{equation}
Comparing~\eqref{eq:MergingRulesSO(N)} and~\eqref{eq:MergingRulesSp(N)}, we find that the rules for merging of $\SOGroup(N)$ and $\SpGroup(N)$ are completely identical, whereas the twisting rules~\eqref{eq:TwistingRulesSO(N)} and~\eqref{eq:TwistingRulesSp(N)} are only \emph{almost} identical, differing by a single sign per equation.

Finally, we find the Casimir eigenvalue by the following calculation:
 \begin{equation}
 \Delta g_{il} = (J_{ij} J_{jk} - \delta_{ik} \delta_{jj}) g_{kl} = -(1 + 2N) g_{il}.
 \end{equation}

\subsection{Defining representation of \texorpdfstring{$\UGroup(N)$}{U(N)}}
\label{sec:examples:ugroup}
The Lie algebra of \( \UGroup(N) \) is the set of all skew-hermitian \( N \times N \)-matrices and a basis is given by the following elements:
\begin{equation}
	\set*{\tfrac{i}{\sqrt{2}}  \left( E^{ab} + E^{ba} \right)}_{a < b} 
	\cup
	\set*{ \tfrac{1}{\sqrt{2}}  \left( E^{ab} - E^{ba} \right)}_{a < b}
	\\
	\cup
	\set[\Big]{ i E^{aa}}_{1 \leq a \leq N}
\end{equation}
where $E_{ab} \in \mathbb{R}^{N \times N}$ with $(E_{ab})_{rs} = \delta_{ar} \delta_{bs}$.
The symmetric \( \AdAction \)-invariant bilinear form
\begin{equation}
	\kappa(X,Y) := - \tr(XY), \quad \forall X,Y \in \UAlgebra(N),
\end{equation}
is positive-definite because \( \kappa(X, X) = \sum_{ij} \abs{X_{ij}}^2 \).
Note that \( \kappa \) is not a scalar multiple of the Killing form of \( \UGroup(N) \) as the latter is degenerate.
The basis introduced above is orthonormal with respect to \( \kappa \).

A direct calculation shows that the completeness relation for \( \UAlgebra(N) \) reads
\begin{equation}
	K_{ijkl} = - \delta_{il} \delta_{jk}.
\end{equation}
According to~\eqref{eq:MergingRulesGenerally}, the merging of two Wilson loops $W_l(g) = \tr(C g^{\pm 1}), W_{l'}(g) = \tr(D g^{\pm 1})$ with $C,D \in \UGroup(N)$ is thus given by
\begin{equation}
\label{eq:MergingRulesU(N)}
\mathcal{M}(W_l, W_{l'})(g)
	= \begin{cases}
	(+,+):  \quad - \tr(CgDg), \\
	(+,-):  \quad + \tr(C D), \\
	(-,+):  \quad + \tr(C D), \\
	(-,-):  \quad - \tr(C g^{-1} D g^{-1}). \\
	\end{cases}
\end{equation}
Furthermore, by~\eqref{eq:TwistingRulesGenerally}, the twisting of a Wilson loop $W_l(g) = \tr(C g^{\pm 1} D g^{\pm 1})$ with $C,D \in \UGroup(N)$ is:
\begin{equation}
\label{eq:TwistingRulesU(N)}
\mathcal{T}(W_l)(g)
	= \begin{cases}
	(+,+):  \quad - \tr(Cg) \tr(Dg), \\
	(+,-):  \quad + \tr(C) \tr(D) , \\
	(-,+):  \quad + \tr(C) \tr(D), \\
	(-,-):  \quad - \tr(Cg^{-1}) \tr(Dg^{-1}). \\
	\end{cases}
\end{equation}
By~\eqref{eq:EigenvalueK-OperatorGeneral}, for the eigenvalue \( \lambda \) of the Laplace operator, we obtain
\begin{equation}
	\label{eq:EigenvalueK-OperatorU(N)}
	\lambda \delta_{ij} = K_{ikkj} = - N \delta_{ij}.
\end{equation}

We now discuss the definition of the unitary Weingarten map based on our general result \cref{prop:matrixCoefficients:tensorRepSchur}.
For this we first recall a few well-known facts concerning the representation theory of the symmetric group \( S_n \).
By, \eg, \parencite[Section~9.1]{GoodmanWallach2009} the isomorphism classes of irreducible representations \( G_\lambda \) of \( S_n \) are bijectively indexed by partitions \( \lambda \) of \( n \).
Moreover, the group algebra of \( S_n \) decomposes as a direct sum
\begin{equation}
	\label{eq:examples:ugroup:decompGroupAlgebra}
	\C S_n \isomorph \bigoplus_{\lambda \vdash n} \End(G_\lambda)
\end{equation}
of simple algebras, see \parencite[Section~9.3.2]{GoodmanWallach2009}.
Let \( p_\lambda \in \C S_n \) be the minimal central idempotent that under this isomorphism acts by the identity on \( G_\lambda \) and by zero on the other components.

Let \( V = \C^N \) be the fundamental representation of \( G = \UGroup(N) \).
Every permutation \( \sigma \in S_n \) acts on \( V^{\tensorProd n} \) by mapping \( v_1 \tensorProd \dotsb \tensorProd v_n \) to \( v_{\sigma^{-1}(1)} \tensorProd \dotsb \tensorProd v_{\sigma^{-1}(n)} \).
Let \( \tau: \C S_n \to \End(V^{\tensorProd n}) \) be the linear extension of this representation.
A moment's reflection convinces us that \( \tau(\sigma)^* = \tau(\sigma^{-1}) \). 
The Schur--Weyl theorem shows that the image of \( \tau \) coincides with the set of invariants \( \bigl(V^{\tensorProd n} \tensorProd (V^*)^{\tensorProd n}\bigr)^G \isomorph \End_G (V^{\tensorProd n}) \).
Thus we are in the position to apply \cref{prop:matrixCoefficients:tensorRepSchur}.
For this purpose, endow \( \C S_n \) with an inner product by declaring the basis \( \sigma \in S_n \) to be orthonormal.
Then a simple calculation shows that the adjoint of \( \tau \) with respect to the inner pairing \( \scalarProd{S_1}{S_2} = \tr(S_2^* S_1) \) on \( \End(V^{\tensorProd n}) \) is 
\begin{equation}
	\tau^* (S) = \sum_{\varsigma \in S_n} \tr \bigl(\tau(\varsigma^{-1}) S\bigr) \, \varsigma.
\end{equation}
According to \cref{rem:matrixCoefficients:weingartenByDiagonalizing} the Weingarten map can be obtained by diagonalizing the operator
\begin{equation}
	\tau^* \circ \tau (\sigma)
		= \sum_{\varsigma \in S_n} \tr \bigl(\tau(\varsigma^{-1}) \tau(\sigma)\bigr) \, \varsigma
		= \sum_{\varsigma \in S_n} \tr \bigl(\tau(\varsigma^{-1}\sigma)\bigr) \, \varsigma
		= \sigma \sum_{\varsigma \in S_n} \tr \bigl(\tau(\varsigma)\bigr) \, \varsigma \,.
\end{equation}
In fact, it turns out that \( \sum_{\varsigma \in S_n} \tr \bigl(\tau(\varsigma)\bigr) \, \varsigma = \sum_{\lambda \vdash n} k_\lambda \, p_\lambda \) for some constants \( k_\lambda \).
In other words, the decomposition~\eqref{eq:examples:ugroup:decompGroupAlgebra} is the eigenspace decomposition of \( \tau^* \circ \tau \) with \( \set{k_\lambda} \) as the associated eigenvalues.
The above identity can be established in two different ways:
\begin{itemize}
	\item	
		First, we can use Schur--Weyl duality again to express \( \varsigma \mapsto \tr\bigl(\tau(\varsigma)\bigr) \) in terms of the characters \( \chi_\lambda \) of the irreducible representation \( G_\lambda \) and the dimension of the associated Weyl module \( F^N_\lambda \) with highest weight \( \lambda \).
		Then the relation \( p_\lambda = \frac{\dim G_\lambda}{n!} \chi_\lambda \), see \parencite[Theorem~9.3.10]{GoodmanWallach2009}, yields
		\begin{equation}
			\sum_{\varsigma \in S_n} \tr \bigl(\tau(\varsigma)\bigr) \, \varsigma = n! \sum_{\substack{\lambda \vdash n \\ l(\lambda) \leq N}} \frac{\dim F^N_\lambda}{\dim G_\lambda} p_\lambda \,,
		\end{equation}
		where \( l(\lambda) \) is the number of parts of the partition \( \lambda \).
		This is the approach taken by \citeauthor{CollinsSniady2006}, \cf \parencite[Proposition~2.3.2]{CollinsSniady2006}.
	\item
		Secondly, since \( \tau(\varsigma) \) is a permutation matrix, \( \tr\bigl(\tau(\varsigma)\bigr) \) coincides with the dimension of the set of fixed points.
		Thus,
		\begin{equation}
			\sum_{\varsigma \in S_n} \tr \bigl(\tau(\varsigma)\bigr) \, \varsigma = \sum_{\varsigma \in S_n} N^{\sharp \varsigma} \, \varsigma,
		\end{equation} 
		where \( \sharp \varsigma \) is the number of cycles of \( \varsigma \).
		This equality can be further simplified by using the Jucys--Murphy elements \( X_k \).
		In fact, \( \sum_{\varsigma \in S_n} N^{\sharp \varsigma} \, \varsigma = \prod_{k=1}^n (N + X_k) \).
		Now using the fact the Gelfand--Tsetlin vectors indexed by standard Young tableaus are joint eigenvectors of the Jucys--Murphy elements one gets
		\begin{equation}
			\sum_{\varsigma \in S_n} \tr \bigl(\tau(\varsigma)\bigr) \, \varsigma = \sum_{\substack{\lambda \vdash n \\ l(\lambda) \leq N}} \prod_{(i,j) \in \lambda} (N + j - i) \, p_\lambda \,.
		\end{equation}
		Equality of the eigenvalues with the above description follows from the Hook length formula which gives an expression for the dimensions of \( F^N_\lambda \) and \( G_\lambda \).
		Following this route leads to the relation of the Weingarten map with the Jucys--Murphy elements discovered in \parencite[Theorem~1.1]{Novak2010} and \parencite[Proposition~2]{ZinnJustin2010}.
\end{itemize} 

Thus, in summary, \cref{prop:matrixCoefficients:tensorRepSchur} in combination with \cref{rem:matrixCoefficients:weingartenByDiagonalizing} recovers \parencites[Theorem~2.1]{Collins2002}[Corollary~2.4]{CollinsSniady2006} in the following form.

\begin{theorem}
	\label{prop:examples:ugroup:weingarten}
	Let $\rho: \UGroup(N) \to \C^N$ be the fundamental representation of \( \UGroup(N) \).
	For non-negative integers \( n, n' \), define
	\begin{equation}
		T^{n, n'}(S) = \int_G \rho^{\tensorProd n}(g) \circ S \circ \rho^{\tensorProd n'}(g) \dif g 
	\end{equation}
	for \( S \in \Hom\bigl((\C^N)^{\tensorProd n'}, (\C^N)^{\tensorProd n}\bigr) \).
	Then \( T^{n, n'} = 0 \) if \( n \neq n' \), and otherwise
	\begin{equation}
		T^{n, n} = \tau \circ \mathrm{Wg} \circ \tau^*,
	\end{equation}
	where \( \tau \) and \( \tau^* \) are defined above and \( \mathrm{Wg}: \C S_n \to \C S_n \) is given by
	\begin{equation}
		\mathrm{Wg}(\sigma) = \sum_{\substack{\lambda \vdash n \\ l(\lambda) \leq N}} \prod_{(i,j) \in \lambda} (N + j - i)^{-1} \, p_\lambda \,.
	\end{equation}
\end{theorem}

\subsection{Defining representation of \texorpdfstring{$\SUGroup(N)$}{SU(N)}}
In \parencite{Jafarov2016}, Wilson loop identities of the shape of \cref{prop:theoremA} are derived for $\SUGroup(N)$, using Stein's method and many technical, auxiliary lemmas. In comparison, we will see that our framework allows us to drastically reduce the amount of necessary calculation needed to arrive at the same conclusion.
\\
Consider $\SUGroup(N)$, with the Lie algebra $\mathfrak{su}(N)$ of skew-hermitian matrices with vanishing trace, and inner product
\begin{equation}
\kappa(X,Y) := -\tr(XY) \quad \forall X,Y \in \mathfrak{su}(N),
\end{equation} 
which is a renormalization of the Killing form by a factor \( -\tfrac{1}{2N} \).
By \parencite{BertlmannKrammer2008}, an orthonormal basis $\{\xi^a\}$ is given by
	\begin{equation}
	\begin{aligned}
		\left \{ \tfrac{i}{\sqrt{2}}  \left( E^{jk} + E^{kj} \right) 
		\right \}_{j < k} 
		&\cup
		\left \{ \tfrac{1}{\sqrt{2}}  \left( E^{jk} - E^{kj} \right) 
		\right \}_{j < k}
		\\
		&\cup
		\left\{ 
			\tfrac{i}{\sqrt{2 l(l+1)}} \left( \sum_{j = 1}^l (E^{jj} - E^{l+1, l+1} \right)  
		\right\}_{1 \leq l \leq N - 1}.
	\end{aligned}
	\end{equation}
The corresponding completeness relation turns out to be
	\begin{equation}
		\xi^a_{ij} \xi^{a}_{kl} = -\delta_{il} \delta_{jk} + \frac{1}{N} \delta_{ij} \delta_{kl}.
	\end{equation}
Again, the merging of two loops $W_l(g) = \tr(Cg^{\pm 1})$ and $W_{l'} = \tr(D g^{\pm 1})$ turns out to be, depending on the exponents:
  \begin{equation}
    \mathcal{M}(W_l, W_{l'})(g)
      = \begin{cases}
        (+,+):  \quad - \tr(CgDg) + \tfrac{1}{N} \tr(Cg) \tr(Dg), \\
        (+,-):  \quad + \tr(C D) - \tfrac{1}{N} \tr(C g) \tr(Dg^{-1}), \\
        (-,+):  \quad + \tr(C D) - \tfrac{1}{N} \tr(C g^{-1} ) \tr(Dg), \\
        (-,-):  \quad - \tr(g^{-1} C g^{-1} D) + \tfrac{1}{N} \tr(C g^{-1}) \tr(D g^{-1} ). \\
      \end{cases}
  \end{equation}
These expressions are linear combinations of what is in \cite{Jafarov2016} called \emph{positive mergers} $W_{l \oplus_{j,j'} l'}$, \emph{negative mergers} $W_{l \ominus_{j,j'} l'}$ and the product of the unchanged loops $W_l \cdot W_{l'}$.

Similarly, the twisting of a loop $W_l(g) = \tr(C g^{\pm 1} D g^{\pm 1})$ yields:
  \begin{equation}
    \mathcal{T}(W_l)(g)
      = \begin{cases}
        (+,+): \quad  -\tr(Cg) \tr(Dg) + \tfrac{1}{N} \tr(Cg Dg), \\
        (+,-): \quad  +\tr(C) \tr(D) - \tfrac{1}{N} \tr(C g Dg^{-1}), \\
        (-,+): \quad  + \tr(C) \tr(D) - \tfrac{1}{N} \tr(C g^{-1} Dg), \\
        (-,-): \quad  -\tr(g^{-1} C) \tr( g^{-1} D) + \tfrac{1}{N} \tr(g^{-1} C g^{-1} D). \\
      \end{cases}
  \end{equation}

Similarly, these expressions are linear combinations of what Jafarov calls the \emph{splitting} $W_{\times_{j,j'}^1} W_{\times_{j,j'}^2}$ and the original unchanged loop $W_l$. 

Finally, the Casimir eigenvalue $\lambda$ is gotten via

	\begin{equation}
		\lambda g_{ij} \stackrel{!}{=} \Delta g_{ij} 
		= 
		\left(-\delta_{ik} \delta_{ll} + \frac{1}{N} \delta_{ik} \right) g_{kj} 
		= 
		\left(-N + \frac{1}{N} \right) g_{ij},
	\end{equation}

Inserting this information into \cref{prop:theoremA} reproduces a version of \cite[Thm 8.1]{Jafarov2016}, the only differences coming from the fact that we used integration by parts \emph{twice} rather than once in our derivation of \cref{prop:theoremA}.

\subsection{Irreducible 7-dimensional representation of \texorpdfstring{$G_2$}{G2}}
\label{sec:examples:g2group}
In all previous examples, the merging and twisting of two Wilson loops were polynomials of Wilson loops again. This is not true in full generality, as we will demonstrate by considering the exceptional, real, compact Lie group $G_2$.
\\
This Lie group is of dimension 14, and its smallest nontrivial irreducible representation $(\rho , V)$ is of dimension 7. We cite from \cite{Schwarz1988} a construction of this representation based on the octonions $\mathbb{O}$.
Recall the 8-dimensional, real, non-associative division algebra given by the octonions $\mathbb{O}$. With respect to the standard basis $\{e_i\}_{i=0}^7$ of $\mathbb{O}$, the multiplication is specified by 
\begin{align}
e_i e_j
= 
\begin{cases}
e_i &\text{ if } j = 0, \\
e_j &\text{ if } i = 0, \\
- \delta_{ij} e_0 + \psi_{ijk} e_k &\text{ else,} \\
\end{cases}
\end{align}
where $\{\psi_{ijk}\}_{i,j,k \in \{1,\dots,7\}}$ denotes a totally antisymmetric symbol, assuming the value 1 on the ordered triples
\begin{equation}
(i,j,k) = 
(1,2,3), \,
(1, 4, 7), \,
(1, 6, 5), \,
(2, 4, 6), \,
(2, 5, 7), \,
(3, 5, 4), \,
(3, 6, 7),
\end{equation}
and zero on all $(i,j,k)$ which do not arise from the above triples by permutation. This algebra admits a linear involution by extension of
\begin{align}
\overline{e_i} = 
\begin{cases}
e_0 &\text{ if } i = 0, \\
-e_i &\text{ if } i > 0,
\end{cases}
\end{align}
and a linear tracial map $\tau_{\mathbb{O}} : \mathbb{O} \to \mathbb{R}$ by extension of
\begin{align}
\tau_{\mathbb{O}}(e_i) := \delta_{i0}.
\end{align}
Now $G_2$ is the Lie group of algebra automorphisms of $\mathbb{O}$. Set $V := \ker \tau_{\mathbb{O}}$, then $G_2$ acts irreducibly and unitarily on $V$ with respect to the inner product
\begin{align*}
B(x,y) := \tau_{\mathbb{O}}(\overline{x}y) \quad \forall x ,y \in V.
\end{align*}
In \cite{Macfarlane2001}, an explicit basis $\{H^a\}_{a = 1,\dots, 14}$ of $\mathfrak{g}_2 := \text{Lie}(G_2)$ as a subalgebra of $\mathfrak{su}(7)$ is constructed, fulfilling the following\footnote{Our choice of generators $\{H^a\}$ differs from the ones in \cite{Macfarlane2001} by a factor of $i / \sqrt{2}$, in order to achieve orthonormality and since we need to view $\mathfrak{su}(7)$ as a Lie algebra of \emph{skew-hermitian} rather than hermitian matrices to achieve the right behaviour under exponentiation.}:
\begin{equation}
\tr(H^a H^b) = \delta_{ab}, \quad (\overline{H^a})^T = -H^a, \quad \forall a,b \in \{1,\dots,14\}.
\end{equation}
These generators constitute an orthonormal basis with respect to the inner product
\begin{equation}
\kappa(X,Y) := \tr(\overline{X}^T Y).
\end{equation}
By \cite{Macfarlane2001}, the completeness relation of $\mathfrak{g}_2$ is then equal to
\begin{equation}
(H^a)_{ij} (H^a)_{kl} = 
\frac{1}{2} 
\left(
\delta_{ik} \delta_{jl}
- 
\delta_{il} \delta_{jk}
\right)
- 
\frac{1}{6} \psi_{rij} \psi_{rkl}.
\end{equation}
In the basis $\{e_i\}_{i \in \{1,\dots,7\}}$ of $V$, define for all $r \in \{ 1,\dots,7 \}$ the endomorphism $\Psi_r := (\psi_{rij})_{i,j \in \{ 1,\dots, 7 \}} \in \End(V)$. Then, the merging of two Wilson loops $W_l(g) = \tr(C g^{\pm 1}), W_{l'}(g) = \tr(D g^{\pm 1})$ with $C,D \in G_2$ can be expressed as
  \begin{equation}
  \label{eq:MergingRulesG_2}
  \begin{aligned}
    &\mathcal{M}(W_l, W_{l'})(g) =
    \\
      &\begin{cases}
        (+,+):  \; \frac{1}{2} \left( \tr(C D^{-1}) - \tr(CgDg)\right) 
        - \tfrac{1}{6} \tr(Cg \Psi_r) \tr(Dg \Psi_r) , \\
        (+,-):  \; \frac{1}{2} \left(\tr(C D) - \tr(CgD^{-1}g)\right) 
        - \tfrac{1}{6} \tr(Cg \Psi_r) \tr( D \Psi_r g^{-1}), \\
        (-,+):  \; \frac{1}{2} \left(\tr(C D) - \tr( C g^{-1} D^{-1} g^{-1})\right) 
        - \tfrac{1}{6} \tr(C\Psi_r g^{-1}) \tr(Dg \Psi_r), \\
        (-,-):  \; \frac{1}{2} \left(\tr(C^{-1} D) - \tr(C g^{-1} D g^{-1})\right) 
        - \tfrac{1}{6} \tr(C\Psi_r g^{-1}) \tr(D \Psi_r g^{-1}).
      \end{cases} 
  \end{aligned}
  \end{equation}
  Here and in the following, summation over the common index $r$ is understood.
\\
The twisting of $W_l(g) = \tr(C g^{\pm 1} D g^{\pm 1})$ with $C,D \in G_2$ is given by
  \begin{equation}
  \begin{aligned}
  \label{eq:TwistingRulesG_2}
    &\mathcal{T}(W_l)(g) =
    \\
      & \begin{cases}
        (+,+):  \; \frac{1}{2} \left( \tr(C D^{-1}) - \tr(Cg) \tr(Dg) \right)
        - \frac{1}{6} \tr(C g \Psi_r ) \tr(D g \Psi_r), \\
        (+,-):  \;  \frac{1}{2} \left( -  \tr(g^{-1} C g D^{-1}) + \tr(C) \tr(D)  \right)
        - \frac{1}{6} \tr(C g \Psi_r ) \tr(D \Psi_r g^{-1}) , \\
        (-,+):  \;  \frac{1}{2} \left( - \tr(C g^{-1} D^{-1} g) + \tr(C) \tr(D)  \right)
        - \frac{1}{6} \tr(C \Psi_r g^{-1}) \tr(D g \Psi_r), \\
        (-,-):  \; \frac{1}{2} \left( \tr(C D^{-1}) - \tr(g^{-1} C) \tr(g^{-1} D) \right) 
        - \frac{1}{6} \tr(C \Psi_r g^{-1}) \tr(D \Psi_r g^{-1}), \\
      \end{cases}
  \end{aligned}	
  \end{equation}
It seems unlikely that the expressions involving the matrices $\Psi_r$ can be simplified any further, and as such, we do not have a polynomial of Wilson loops, but only of generalized Wilson loops in the sense of \cref{def:GeneralizedWilsonLoops}.
\\
We can also apply \cref{prop:matrixCoefficients:tensorRepSchur} and \cref{rem:matrixCoefficients:weingartenByDiagonalizing} to $G_2$ to calculate certain Weingarten functions: Denote by $G_2^\C$ the complex Lie group given by the automorphisms of the complexified octonion algebra. It also assumes an irreducible, unitary representation on the complexification $V_\C = (\ker \tau_{\mathbb{O}})_\C$ and in \cite{Schwarz1988}, generators for invariant spaces $(V_\C ^{\otimes n} \otimes_\C (V_\C^*)^{\otimes n'})^{G_2^\C}$ have been calculated. Due to connectedness of $G_2^\C$, we have for every $G_2^\C$-module $M$
\begin{align*}
M^{G_2^\C}
=
M^{\mathrm{Lie}(G_2^\C)}
=
M^{(\mathfrak{g}_2)_\C},
\end{align*}
denoting by $(\mathfrak{g}_2)_\C$ the complexification of the Lie algebra $\mathfrak{g}_2$. As such, we can deduce from this the $G_2$ invariants $(V ^{\otimes n} \otimes (V^*)^{\otimes n'})^{G_2} = (V ^{\otimes n} \otimes (V^*)^{\otimes n'})^{\mathfrak{g}_2}$.
\\
The representation space $V$ admits an invariant, non-degenerate bilinear form
\begin{align*}
\alpha : V \otimes V \to \mathbb{R}, \quad e_i \otimes e_j \mapsto -\tau_ {\mathbb{O}}(e_i e_j),
\end{align*}
hence $V \cong V^*$ as $G_2$-modules and one may restrict to the invariants $(V^{\otimes n})^{G_2}$ without loss of generality.
\\
For simplicity, let us study the case $n = 2$. An analysis of $n > 2$ is possible, but finding a basis of $(V^{\otimes n})^{G_2}$ becomes more difficult due to the presence of nontrivial relations between the generators determined in \cite{Schwarz1988}.
Now, $(V \otimes V)^{G_2}$ is one-dimensional and is generated by $u := \sum_{i=1}^7 e_i \otimes e_i$. This element is dual to $\alpha \in V^* \otimes V^*$ due to the relation $\tau_\mathbb{O}(e_i e_j) = -\delta_{ij}$ for $1 \leq i,j \leq 7$.
We set $\mathcal{A} = \mathbb{R} u \subset V \otimes V$, equip this subspace with the inner product $\langle u, u \rangle = 1$, and define $\tau : \mathcal{A} \to V \otimes V$ to be the embedding of this subspace.
The adjoint $\tau^*$ with respect to the scalar product on $V \otimes V$ defined in \cref{prop:matrixCoefficients:tensorRepSchur} is given by
\begin{align*}
\tau^* (e_i \otimes e_j) = \delta_{ij} u, \quad \forall 1 \leq i,j \leq 7,
\end{align*}
so $(\tau^* \tau) (u) = 7 u$.
Hence, by \cref{rem:matrixCoefficients:weingartenByDiagonalizing}, we find 
\begin{equation}
	\mathrm{Wg} = \frac{1}{7} \id_{\mathcal{A}}.
\end{equation}
Now, \cref{prop:matrixCoefficients:tensorRepSchur} allows us to deduce that for all $1 \leq i, j \leq 7$ we have
\begin{align*}
\int_{G_2} \rho^{\otimes 2}(g) (e_i \otimes e_j) \dif g 
=
(\tau \circ \mathrm{Wg} \circ \tau^*) (e_i \otimes e_j) =
\frac{1}{7} \delta_{ij} \sum_{k=1}^7 e_k \otimes e_k
\end{align*}
or, as a scalar integral,
\begin{align*}
\int_{G_2} 
\rho(g)_{i_1 j_1}
\rho(g)_{i_2 j_2}
\dif g
=
\frac{1}{7}
\delta_{i_1 i_2}
\delta_{j_1 j_2}
.
\end{align*}

\begin{refcontext}[sorting=nyt]{}
  \printbibliography
\end{refcontext}
\end{document}